\documentclass[12pt]{amsart}
\usepackage{amsmath, amssymb, comment, graphicx, url, amsthm, multicol, latexsym, enumerate, bbm, mathtools, physics, microtype, cite, xcolor, float}
\usepackage{units}
\usepackage[all,cmtip]{xy}
\usepackage{wrapfig}
\usepackage[hmargin=1in]{geometry} 
\usepackage{tikz}
\usepackage{tikzsymbols}
\usetikzlibrary{shapes.geometric, arrows}
\usepackage{enumitem}
\usepackage[cm]{sfmath}
\usepackage{hyperref}
\hypersetup{
    colorlinks=true,
    linkcolor=blue,
    citecolor=magenta,
    filecolor=magenta,      
    urlcolor=magenta,
}
\usepackage{comment}

\definecolor{andresblue}{rgb}{0,0.72,0.92}
\definecolor{andrespink}{rgb}{1,0,1}

\newtheorem{theorem}{Theorem}[section]

\newtheorem{conjecture}[theorem]{Conjecture}
\newtheorem{construction}[theorem]{Construction}
\newtheorem{lemma}[theorem]{Lemma}
\newtheorem{corollary}[theorem]{Corollary}

\theoremstyle{definition}
\newtheorem{remark}[theorem]{Remark}
\newtheorem{definition}[theorem]{Definition}

\newtheorem{example}[theorem]{Example}

\theoremstyle{remark}

\newcommand{\R}{\mathbb{R}}

\newcommand{\Z}{\mathbb{Z}}
\newcommand{\N}{\mathbb{N}}

\newcommand{\bv}{\mathbf{v}}

\newcommand{\bw}{\mathbf{w}}

\newcommand{\ba}{\mathbf{a}}

\newcommand{\bx}{\mathbf{x}}

\newcommand{\be}{\mathbf{e}}
\newcommand{\bq}{\mathbf{q}}

\newcommand{\bl}{\mathbf{l}}
\newcommand{\bs}{\mathbf{s}}

\newcommand{\vol}{\mathrm{vol}}

\newcommand{\calP}{\mathcal{P}}
\newcommand{\calO}{\mathcal{O}}
\newcommand{\calM}{\mathcal{M}}
\newcommand{\calD}{\mathcal{D}}
\newcommand{\calC}{\mathcal{C}}

\newcommand{\defterm}[1]{\emph{#1}}

\newcommand{\precdot}{\prec\mathrel{\mkern-5mu}\mathrel{\cdot}}

\makeatletter
\newcommand{\preceqdot}{\mathrel{\mathpalette\pr@ceqd@t\relax}}
\newcommand{\pr@ceqd@t}[2]{%
  \begingroup
  \sbox\z@{$#1\prec$}\sbox\tw@{$#1\preceq$}%
  \dimen@=\dimexpr\ht\tw@-\ht\z@\relax
  {\preceq}%
  \mkern-5mu
  \raisebox{\dimen@}{$\m@th#1\cdot$}%
  \endgroup
}
\makeatother

\makeatletter 
\newtheorem*{rep@theorem}{\rep@title}\newcommand{\newreptheorem}[2]{%
\newenvironment{rep#1}[1]{%
\def\rep@title{\bf #2 \ref{##1}}%
\begin{rep@theorem}}%
{\end{rep@theorem}}}
\makeatother
\newreptheorem{theorem}{Theorem}

\newtheorem*{rep@proposition}{\rep@title}\newcommand{\newrepproposition}[2]{%
\newenvironment{rep#1}[1]{%
\def\rep@title{\bf #2 \ref{##1}}%
\begin{rep@proposition}}%
{\end{rep@proposition}}}
\makeatother
\newreptheorem{proposition}{Proposition}

\usepackage[colorinlistoftodos]{todonotes}

\definecolor{munsell}{rgb}{0.0, 0.5, 0.69}

\begin{document}


\title{Generalized snake posets, order polytopes, and  lattice-point enumeration}

\author{Eon Lee}
\address{\scriptsize{School of Electrical Engineering and Computer Science, Gwangju Institute of Science and Technology}}
\email{\scriptsize{eonlee1125@gmail.com}}

\author{Andr\'es R. Vindas-Mel\'endez}
\address{\scriptsize{Department of Mathematics, Harvey Mudd College}, \url{https://math.hmc.edu/arvm}}
\email{\scriptsize{avindasmelendez@g.hmc.edu}}

\author{Zhi Wang}
\address{\scriptsize{Hal\i c\i o\u{g}lu Data Science Institute, University of California, San Diego}}
\email{\scriptsize{wangzhi0467@outlook.com}}


\begin{abstract}
   Building from the work of von Bell et al.~(2022), we study the Ehrhart theory of order polytopes arising from a special class of distributive lattices, known as generalized snake posets. 
   We present arithmetic properties satisfied by the Ehrhart polynomials of order polytopes of generalized snake posets along with a computation of their Gorenstein index. 
   Then we give a combinatorial description of the chain polynomial of generalized snake posets as a direction to obtain the $h^*$-polynomial of their associated order polytopes.
   Additionally, we present explicit formulae for the $h^*$-polynomial of the order polytopes of the two extremal examples of generalized snake posets, namely the ladder and regular snake poset. 
   We then provide a recursive formula for the $h^*$-polynomial of any generalized snake posets and show that the $h^*$-vectors are entry-wise bounded by the $h^*$-vectors of the two extremal cases. 
\end{abstract}


\maketitle

\section{Introduction}
This paper investigates the lattice-point enumeration and related properties of the order polytopes of a class of distributive lattices, known as generalized snake posets. 
Order polytopes were popularized by Richard Stanley through his seminal paper \cite{StanleyTwoPosetPolytopes}, but have ``been the subject of considerable scrutiny" even before that, though most results were ``scattered throughout the literature." 
For example, some of the first works on order polytopes include \cite{Dobbertin} and \cite{Geissinger}.
Ever since Stanley's paper, order polytopes have actively drawn mathematician's attention, where some aspects studied include: geometric and algebraic properties~\cite{DoignonRexhep, HaaseKohlTsuchiya, HibiLiLiMuTsuchiya, HibiMatsuda, HibiMatsudaKazunoriOhsugiHidefumiShibata}, connections between flow polytopes and order polytopes~\cite{LiuMeszarosStDizier, MeszarosMoralesStriker}, and lattice-point enumeration~\cite{ChappellFriedlSanyal, LiuTsuchiya}.
One of Stanley's fundamental observations is that the arrangement given by all hyperplanes of the form $x_i=x_j$ for $i\neq j$ induces a regular unimodular triangulation of the order polytope for any poset, known as the \emph{canonical triangulation} of an order polytope.

Generalized snake posets are posets that are constructed recursively by adding a square face at the bottom and gluing it to an edge of the previous square.
These special posets make an appearance in \cite[Section 3]{StanleyFlag}, where enumerative results regarding their flag $h$-vectors being  multiplicity-free (i.e., assumes only the values $0$ and $\pm1$) and bounds on the number of their linear extensions arise as consequences of the results presented. 
A geometric approach to studying these special posets can be considered through their order polytopes, which was the focus of \cite{vonBell+}. 
There the authors studied the circuits, flips, regular triangulations, and volumes of the order polytopes of generalized snake posets. 

Our contributions are on order polytopes of generalized snake posets, where we take particular focus on their Ehrhart theory.
This work can be considered as initiating the study of order polytopes of distributive lattices in the context of Ehrhart theory. 
The main object of interest is a distributive lattice with an upper bound $r$ on the number of elements in a single rank. 
In the current work, we completely formulate the case for $r = 2$, as knowing the $h^*$-polynomial for every generalized snake posets is enough to compute the Ehrhart polynomial for any distributive lattices with at most $2$ elements in each rank. 
This follows from the fact that generalized snake posets are the unique distributive lattice with exactly $2$ elements in each rank other than rank $1$ and the highest rank \cite{StanleyFlag}.
This, together with the fact that the $h^*$-polynomial of the order polytope of a chain is known, and the fact that the $h^*$-polynomial of the order polytope of $P \oplus Q$ can be obtained by knowing the $h^*$-polynomial of the order polytope of $P$ and $Q$, tells us that the knowledge we present is sufficient.

The article is organized as follows. 
\begin{itemize}
    \item In Section \ref{sec:prelims} we present background and preliminaries on triangulations and order polytopes.
    Moreover, Theorem \ref{thm:roots} presents arithmetic properties that the Ehrhart polynomials of order polytopes of generalized snake posets must satisfy, including a computation of their Gorenstein index. 
    
    \item Section \ref{sec:chain_polynomial} focuses on a combinatorial description of the chain polynomial of generalized snake posets as a means to obtaining the $h^*$-polynomial of an order polytope of a generalized snake poset.
    In particular, Theorem \ref{thm:chain_coefficients} presents the combinatorial formula of the chain polynomial of any generalized snake poset. 
    
    \item Then in Section \ref{sec:lattice-path} we explore the $h^*$-polynomials of generalized snake posets via lattice-path enumeration. 
    Through this lattice-path perspective, we obtain closed formulas $h^*$-polynomials for the regular snake (Theorem \ref{thm: h*_snake}) and the ladder (Theorem \ref{thm:h*_ladder}). 
    Additionally, we obtain a recurrence formula for the $h^*$-polynomial of any generalized snake (Theorem \ref{thm: recurrence for h*}), where the base cases are in fact the $h^*$-polynomials of the ladder.
    Furthermore, a coefficient-wise monotonicity result for the $h^*$-polynomial of any generalized snake is given (Theorem \ref{thm: coef mono}), bounding the $h^*$-vectors of any generalized snake posets with the $h^*$-vectors of the ladder and the regular snake (Corollary \ref{cor 4.20}). 
    
    \item We conclude with conjectures for further research in Section \ref{sec:conclusion} regarding the real rooted-ness of $h^*$-polynomials and Ehrhart positivity.
\end{itemize}


\section{Background \& Preliminaries}\label{sec:prelims}
Our main focus of study is the lattice-point enumeration of  polytopes that arise from a special partially ordered set (poset), namely a generalized snake poset, which we now present.

\subsection{Generalized snake posets}
The family of generalized snake posets $P(\bw)$ are distributive lattices with width two and were studied geometrically in \cite{vonBell+}.

\begin{definition}[Definition 3.1, \cite{vonBell+}]
For $n\in \mathbb{Z}_{\geq0}$, a \emph{generalized snake word} is a word of the form $\bw=w_0 w_1 \cdots w_n$ where $w_0 =\varepsilon$ is the empty letter and $w_i$ is in the alphabet $\{L,R\}$ for $i=1,\ldots, n$.
The \emph{length} of the word is $n$, which is the number of letters in $\{L,R\}$.
\end{definition}

\begin{definition}[Definition 3.2, \cite{vonBell+}]
Given a generalized snake word $\bw=w_0w_1\cdots w_n$, the \emph{generalized snake poset} $P(\bw)$ is recursively defined as follows:
\begin{itemize}
    \item $P(w_0) = P(\varepsilon)$ is the poset on elements $\{0,1,2,3\}$ with cover relations $1\prec 0$, $2\prec 0$, $3\prec 1$ and $3\prec 2$. 

    \item $P(w_0w_1\cdots w_n)$ is the poset $P(w_0w_1\cdots w_{n-1}) \cup \{2n+2,2n+3\}$ with the added cover relations $2n+3 \prec 2n+1$, $2n+3 \prec 2n+2$, and 
    $$\begin{cases}
    2n+2 \prec 2n-1, 
        & \text{ if } n=1 \text{ and } w_n = L, \text{ or } n \geq 2 \text{ and }  w_{n-1}w_n \in \{RL,LR\},\\
    2n+2 \prec 2n, 
        & \text{ if } n=1 \text{ and } w_n = R, \text{ or } n \geq 2 \text{ and }  w_{n-1}w_n \in \{LL,RR\}.
    \end{cases}$$

\end{itemize}
In this definition, the minimal element of the poset $P(\bw)$ is $\widehat0=2n+3$, and the maximal element of the poset is $\widehat1 = 0$.
\end{definition} 

If $\bw=w_0w_1\cdots w_n$ is a generalized snake word of length $n$, then $P(\bw)$ is a distributive lattice of width two and rank $n+2$.
Two special cases of generalized snake posets are the \emph{regular snake poset} and the \emph{ladder poset}.
For the length $n$ word $\varepsilon LRLR\cdots$, $S_n:=P(\varepsilon LRLR\cdots)$ is the \emph{regular snake poset}, and for the length $n$ word $\varepsilon LLLL\cdots$, $\mathcal{L}_n:=P(\varepsilon LLLL\cdots)$ is the \emph{ladder poset}.

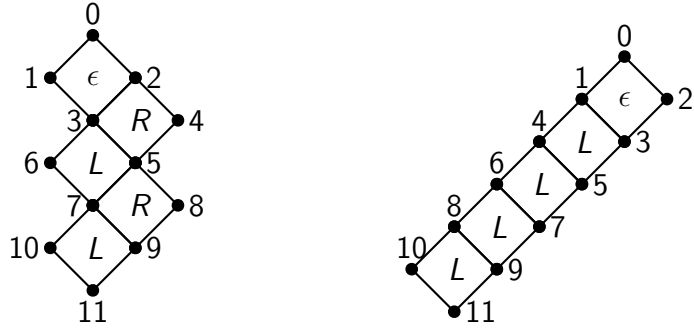
\begin{figure}[htbp]
    \centering
    \begin{tikzpicture}[scale= 0.8, rotate = -45]
    \def\xmin{-5}
    \def\xmax{5}
    \def\ymin{0}
    \def\ymax{10}

    \draw[black,thick] (0,0) rectangle (1,1);
    \foreach \x/\y in {0/0, 1/0, 0/1, 1/1}
        \fill (\x,\y) circle (3pt);
    \draw[black,thick] (0,1) rectangle (1,2);
    \foreach \x/\y in {0/1, 1/1, 0/2, 1/2}
        \fill (\x,\y) circle (3pt);
    \draw[black,thick] (0,1) rectangle (-1,2);
    \foreach \x/\y in {0/1, 0/2, -1/1, -1/2}
        \fill (\x,\y) circle (3pt);
    \draw[black,thick] (-1,2) rectangle (0,3);
    \foreach \x/\y in {-1/2, 0/2, -1/3, 0/3}
        \fill (\x,\y) circle (3pt);
    \draw[black,thick] (-1,2) rectangle (-2,3);
    \foreach \x/\y in {-1/2, -1/3, -2/2, -2/3}
        \fill (\x,\y) circle (3pt);

    \draw (-1.7, 2.3) node[right] {$\epsilon$};
    \draw (-0.7, 2.3) node[right] {$R$};
    \draw (-0.7, 1.3) node[right] {$L$};
    \draw (0.3, 1.3) node[right] {$R$};
    \draw (0.3, 0.3) node[right] {$L$};

    \draw (-2, 3) node[above] {$0$};
    \draw (-2, 2) node[left] {$1$};
    \draw (-1, 3) node[right] {$2$};
    \draw (-1, 2) node[left] {$3$};
    \draw (0, 3) node[right] {$4$};
    \draw (0, 2) node[right] {$5$};
    \draw (-1, 1) node[left] {$6$};
    \draw (0, 1) node[left] {$7$};
    \draw (1, 2) node[right] {$8$};
    \draw (1, 1) node[right] {$9$};
    \draw (0, 0) node[left] {$10$};
    \draw (1, 0) node[below] {$11$};
\end{tikzpicture}
\hspace{2 cm}
    \begin{tikzpicture}[scale= 0.8, rotate = -45]
    \def\xmin{-5}
    \def\xmax{5}
    \def\ymin{0}
    \def\ymax{10}

    \draw[black,thick] (0,0) rectangle (1,1);
    \foreach \x/\y in {0/0, 1/0, 0/1, 1/1}
        \fill (\x,\y) circle (3pt);
    \draw[black,thick] (0,1) rectangle (1,2);
    \foreach \x/\y in {0/1, 1/1, 0/2, 1/2}
        \fill (\x,\y) circle (3pt);
    \draw[black,thick] (0,2) rectangle (1,3);
    \foreach \x/\y in {0/2, 1/2, 0/3, 1/3}
        \fill (\x,\y) circle (3pt);
    \draw[black,thick] (0,3) rectangle (1,4);
    \foreach \x/\y in {0/3, 1/3, 0/4, 1/4}
        \fill (\x,\y) circle (3pt);
    \draw[black,thick] (0,4) rectangle (1,5);
    \foreach \x/\y in {0/4, 1/4, 0/5, 1/5}
        \fill (\x,\y) circle (3pt);

    \draw (0.3, 4.3) node[right] {$\epsilon$};
    \draw (0.3, 3.3) node[right] {$L$};
    \draw (0.3, 2.3) node[right] {$L$};
    \draw (0.3, 1.3) node[right] {$L$};
    \draw (0.3, 0.3) node[right] {$L$};

    \draw (0, 5) node[above] {$0$};
    \draw (0, 4) node[above] {$1$};
    \draw (0, 3) node[above] {$4$};
    \draw (0, 2) node[above] {$6$};
    \draw (0, 1) node[above] {$8$};
    \draw (0, 0) node[above] {$10$};
    
    \draw (1, 5) node[right] {$2$};
    \draw (1, 4) node[right] {$3$};
    \draw (1, 3) node[right] {$5$};
    \draw (1, 2) node[right] {$7$};
    \draw (1, 1) node[right] {$9$};
    \draw (1, 0) node[right] {$11$};
    
\end{tikzpicture}
    \caption{The regular snake poset $P(\epsilon RLRL)$ and the ladder poset $P(\epsilon LLLL)$}
    \label{fig: what's GSP}
\end{figure}
The labeling of $J(P(\bw))$ is given by the following lemma.
\begin{lemma} [Lemma 3.3, \cite{vonBell+}] \label{lemma: the labeling of J(P)}
    Let $w = w_0 w_1 \cdots w_n$ be a generalized snake word. If $k \geq 0$ is the largest index such that $w_k \neq w_n$, then $J(P(w)) =$
\begin{align*}
J(P(w_0 w_1 \cdots w_{n-1})) \cup \{\langle 2n + 3 \rangle, \langle 2n + 2 \rangle, \langle 2n + 2, 2k + 2 \rangle\} \cup \{\langle 2n + 2, 2k + 2i + 1 \rangle\}_{i=1}^{n-k}.
\end{align*}
\end{lemma}

To assist with the visualization of $J(P)$ and proofs to come, we define the embedding of $J(P)$ onto $\Z^2$ through the following construction.
\begin{construction}\label{embedding}
    Let $\bw = w_0 w_1 \cdots w_n$ be a generalized snake word and $k \geq 0$ the largest index such that $w_k \neq w_n$. 
    Then $J(P(\bw))$ is recursively embedded onto the integer lattice $\Z^2$ as follows:

\begin{enumerate}
    \item \textbf{[Base Step]} $J(P(w_0)) = J(P(\epsilon))$ is embedded onto $\Z^2$ 
    following description (1a) if $w_1 = R$, and following description (1b) if $w_1 = L$.
    (If $n = 0$, follow any of the $2$ descriptions.)

    \begin{enumerate}
        \item $\emptyset$ to $(1, 2)$, $\langle 0 \rangle$ to $(1, 1)$, $\langle 1 \rangle$ to $(0, 1)$,  $\langle 2 \rangle$ to $(1, 0)$,  $\langle 1, 2 \rangle$ to $(0, 0)$, and  $\langle 3 \rangle$ to $(-1, 0)$, or
        
        \item $\emptyset$ to $(2, 1)$, $\langle 0 \rangle$ to $(1, 1)$, $\langle 1 \rangle$ to $(0, 1)$,  $\langle 2 \rangle$ to $(1, 0)$,  $\langle 1, 2 \rangle$ to $(0, 0)$, and  $\langle 3 \rangle$ to $(0, -1)$.
    \end{enumerate}

    \item \textbf{[Recursive Step]} To embed $J(P(w_0, \dots, w_n))$ onto $\Z^2$, shift the embedding of \\ $J(P(w_0, \dots, w_{n-1}))$ by the vector $(1, 1)$.\footnote{We include this first step to assure the poset is embedded onto the two upper quadrants.} 
    Then $J(P(\bw))$ is embedded onto $\Z^2$ by the following:

    \begin{enumerate}
        \item if $w_n=R$, embed $\langle 2n+3 \rangle$ to $(-1, 0)$, $\langle 2n+2 \rangle$ to $(n-k + 1, 0)$, $\langle 2n+2, 2k+2 \rangle$ to $(n-k, 0)$, and $\langle 2n+2, 2k+2i+1\rangle$ to $(n-k-i, 0)$ for all $i \in [n-k]$, or
        
        \item if $w_n=L$, embed $\langle 2n+3 \rangle$ to $(0, -1)$, $\langle 2n+2 \rangle$ to $(0, n-k + 1)$, $\langle 2n+2, 2k+2 \rangle$ to $(0, n-k)$, and $\langle 2n+2, 2k+2i+1\rangle$ to $(0, n-k-i)$ for all $i \in [n-k]$.

\end{enumerate}
\end{enumerate}
\end{construction}

\begin{remark}
For the remainder of the paper, when coordinates in $\Z^2$ are associated to elements of $J(P(\bw))$ for a generalized snake word $\bw$, we are working with the assumption that $J(P(\bw))$ is embedded onto $\Z^2$ as described in Construction \ref{embedding}.    
\end{remark}

As an example, $J(P(\epsilon LRR))$ and $J(P(\epsilon RRR))$ are shown in Figure \ref{fig: J(P)}.

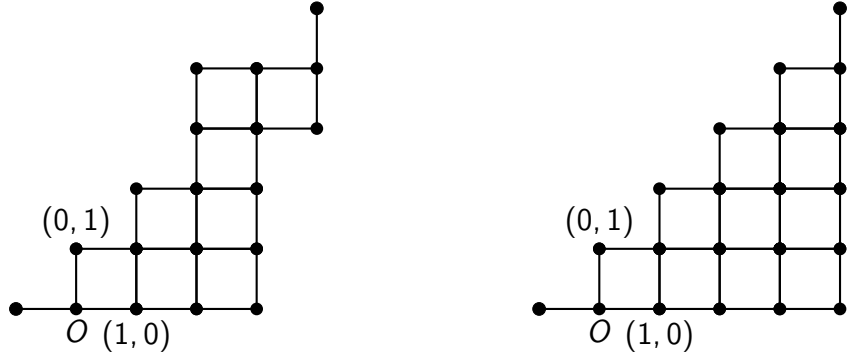
\begin{figure}[htbp]
    \centering
    \begin{tikzpicture}[scale=0.8]
    \def\xmin{-1}
    \def\xmax{4}
    \def\ymin{0}
    \def\ymax{5}
    
    \draw (0, 0) node[below] {$O$};
    \draw (1, 0) node[below] {$(1, 0)$};
    \draw (0, 1) node[above] {$(0, 1)$};

    \draw[black,thick] (0,0) rectangle (1,1);
    \foreach \x/\y in {0/0, 1/0, 0/1, 1/1}
        \fill (\x,\y) circle (3pt);
    \draw[black,thick] (1,0) rectangle (2,1);
    \foreach \x/\y in {1/0, 2/0, 1/1, 2/1}
        \fill (\x,\y) circle (3pt);
    \draw[black,thick] (2,0) rectangle (3,1);
    \foreach \x/\y in {2/0, 3/0, 0/1, 3/1}
        \fill (\x,\y) circle (3pt);
    \draw[black,thick] (1,1) rectangle (2,2);
    \foreach \x/\y in {1/1, 2/1, 1/2, 2/2}
        \fill (\x,\y) circle (3pt);
    \draw[black,thick] (2,1) rectangle (3,2);
    \foreach \x/\y in {2/1, 3/1, 2/2, 3/2}
        \fill (\x,\y) circle (3pt);
    \draw[black,thick] (2,2) rectangle (3,3);
    \foreach \x/\y in {2/2, 3/2, 2/3, 3/3}
        \fill (\x,\y) circle (3pt);
    \draw[black,thick] (2,3) rectangle (3,4);
    \foreach \x/\y in {2/3, 3/3, 2/4, 3/4}
        \fill (\x,\y) circle (3pt);
    \draw[black,thick] (3,3) rectangle (4,4);
    \foreach \x/\y in {3/3, 4/3, 3/4, 4/4}
        \fill (\x,\y) circle (3pt);
    
    \draw[black, thick] (-1, 0) -- (0 , 0);
    \draw[black, thick] (4, 4) -- (4 , 5);
    \filldraw (-1,0) circle (3pt);
    \filldraw (4, 5) circle (3pt);
\end{tikzpicture}
\hspace{2.5 cm}
    \begin{tikzpicture}[scale=0.8]
    \def\xmin{-1}
    \def\xmax{4}
    \def\ymin{0}
    \def\ymax{5}
    
    \draw (0, 0) node[below] {$O$};
    \draw (1, 0) node[below] {$(1, 0)$};
    \draw (0, 1) node[above] {$(0, 1)$};

    \draw[black,thick] (0,0) rectangle (1,1);
    \foreach \x/\y in {0/0, 1/0, 0/1, 1/1}
        \fill (\x,\y) circle (3pt);
    \draw[black,thick] (1,0) rectangle (2,1);
    \foreach \x/\y in {1/0, 2/0, 1/1, 2/1}
        \fill (\x,\y) circle (3pt);
    \draw[black,thick] (2,0) rectangle (3,1);
    \foreach \x/\y in {2/0, 3/0, 0/1, 3/1}
        \fill (\x,\y) circle (3pt);
    \draw[black,thick] (1,1) rectangle (2,2);
    \foreach \x/\y in {1/1, 2/1, 1/2, 2/2}
        \fill (\x,\y) circle (3pt);
    \draw[black,thick] (2,1) rectangle (3,2);
    \foreach \x/\y in {2/1, 3/1, 2/2, 3/2}
        \fill (\x,\y) circle (3pt);
    \draw[black,thick] (2,2) rectangle (3,3);
    \foreach \x/\y in {2/2, 3/2, 2/3, 3/3}
        \fill (\x,\y) circle (3pt);
    \draw[black,thick] (3,3) rectangle (4,4);
    \foreach \x/\y in {3/3, 4/3, 3/4, 4/4}
        \fill (\x,\y) circle (3pt);
    \draw[black,thick] (3,0) rectangle (4,1);
    \foreach \x/\y in {3/0, 3/1, 4/0, 4/1}
        \fill (\x,\y) circle (3pt);
    \draw[black,thick] (3,1) rectangle (4,2);
    \foreach \x/\y in {3/1, 3/2, 4/1, 4/2}
        \fill (\x,\y) circle (3pt);
    \draw[black,thick] (3,2) rectangle (4,3);
    \foreach \x/\y in {3/2, 3/3, 4/2, 4/3}
        \fill (\x,\y) circle (3pt);
    
    \draw[black, thick] (-1, 0) -- (0 , 0);
    \draw[black, thick] (4, 4) -- (4 , 5);
    \filldraw (-1,0) circle (3pt);
    \filldraw (4, 5) circle (3pt);
\end{tikzpicture}
    \caption{ $J(P(\epsilon LRR))$ on the left and $J(P(\epsilon RRR))$ on the right.}
    \label{fig: J(P)}
\end{figure}


\subsection{Order polytopes}\label{subsec:order_polytopes}

Let $P$ be a poset on the set of elements $[d]:=\{1,\ldots,d\}$. 
We abuse notation and write $P$ to denote the elements of $P$.
The \emph{order polytope} of $P$, popularized by Stanley~\cite{StanleyTwoPosetPolytopes}, is defined as
\[
\mathcal{O}(P) = \left\{\bx=(x_1,\ldots, x_d)\in [0,1]^d: x_i\leq x_j \text{ for }i<_Pj\right\}.
\]

An \emph{upper order ideal of $P$}, also called a \emph{filter}, is a set $A\subseteq P$ such that if $i\in A$ and $i<_Pj$, then $j\in A$.
Let $J(P)$ denote the poset of upper order ideals of $P$, ordered by reverse inclusion.
We use $\langle p_1,\dots,p_k\rangle$ to denote the upper order ideal generated by elements $p_1,...,p_k \in P$.
Let $\be_1,\ldots,\be_d$ denote the standard basis vectors of $\R^d$.
For an upper order ideal $A\in J(P)$, define the \emph{characteristic vector} $\bv_A:=\sum_{i\in A}\be_i$.
The vertices of $\mathcal{O}(P)$ are given by
\[
V(\mathcal{O}(P))=\left\{\bv_A:A\in J(P)\right\}.
\]
Define a hyperplane $\mathcal{H}_{i,j}= \{ \bx\in \R^d : x_i = x_j \}$ for $1\leq i < j \leq d$.
The set of all such hyperplanes induces a triangulation $\mathcal{T}$ of $\mathcal{O}(P)$ known as the \emph{canonical triangulation}, which has the following three important properties:
\begin{enumerate}
    \item $\mathcal{T}$ is unimodular,
    \item the maximal simplices are in bijection with the linear extensions of $P$, so the normalized volume of the order polytope is
    $$ \vol (\mathcal{O}(P)) = \# \text{ of linear extensions of } P, \text{ and} $$
    \item the simplex corresponding to a linear extension $(a_1,\ldots,a_d)$ of $P$ is
    $$ \sigma_{a_1,...,a_d} = \left\{\bx \in [0,1]^d : x_{a_1} \leq x_{a_2} \leq \cdots \leq x_{a_d}\right\}, $$
    with vertex set 
    $\{\mathbf{0}, \mathbf{e}_{a_d}, \mathbf{e}_{a_{d-1}}+ \mathbf{e}_{a_d}, \ldots ,  \mathbf{e}_{a_1}+\cdots+\mathbf{e}_{a_d} = \mathbf{1}\}.$
\end{enumerate}


\subsection{Ehrhart theory}
The \defterm{Ehrhart function} of a polytope (i.e., the convex hull of finitely many points)  $P\subset \R^n$,  is the lattice-point enumerator $L_P(t):=|tP\cap \Z^n|$, where $tP=\{t\bx:\ \bx\in P\}$. 
Ehrhart theory, developed by Eug\`ene Ehrhart in the 1960s (see, e.g., \cite{Ehrhart}) can be regarded as a discrete version of integration: the growth of the Ehrhart function provides information about the volume and surface area of $P$.
Ehrhart theory is connected to the combinatorics of simplicial complexes, as well as number theory and discrete analysis; for a comprehensive overview, see~\cite{BeckRobins}.

When $P$ is a lattice polytope (i.e., its vertices have integer coordinates), the Ehrhart function is a polynomial in $t$, with degree equal to the dimension of $P$, leading coefficient equal to its normalized volume, second-leading coefficient equal to half the normalized surface area, and constant coefficient 1.
The other coefficients of $L_P(t)$ are more mysterious, and in general can be negative. 
One of the main problems in Ehrhart theory is to characterize the coefficients that a polynomial must have in order to be the Ehrhart polynomial of a lattice polytope. 
The Ehrhart polynomial of a lattice polytope $P$ of dimension $n$ can always be written in the form $L_P(t)=\sum_{i=0}^nh_i^*\binom{t+n-i}{n}$; the sequence $(h_0^*,\dots,h_n^*)$ is called the \defterm{$h^*$-vector}. 
One can also encode $L_P(t)$ in a generating series, called the \emph{Ehrhart series}
\[
  \operatorname{Ehr}(P;z):=\sum_{t\geq 0}L_P(t) \, z^t=\frac{h^*(P;z)}{(1-z)^{n+1}} \, .
\]
The polynomial $h^*(P;z)=\sum_{i=0}^n h^*_iz^i$ is known as the \emph{$h^*$-polynomial} of~$P$ and is of degree less than $n+1$ and has nonnegative integer coefficients.
Ehrhart and $h^*$-polynomials are  important invariants of lattice polytopes, as they encode much of a polytope's geometry, arithmetic, and combinatorics. 
An important and useful result in Ehrhart theory is known as \emph{Ehrhart--Macdonald reciprocity} \cite{Ehrhart, Macdonald}, which states that for any convex lattice polytope $P$ of dimension $n$:
\begin{equation*}
    L_{P}(-t)=(-1)^nL_{P^{\circ}}(t).
\end{equation*}
This provides a relationship between a polytope and its interior and also gives an interpretation for negative values of $t$.
Evaluating an Ehrhart polynomial at $-t$ gives the lattice-point count for the $t^{\text{th}}$ dilate of its interior (up to a sign).

For simplicity, denote the $h^*$-polynomial of $\mathcal{O}(P(\bw))$ as $h^*(\bw; z)$, i.e., \[h^*(\bw; z):= h^*(\mathcal{O}(P(\bw)); z).\]
Similarly, the Ehrhart polynomial will be denoted as $L(\bw; t)$. 

\subsection{Order polytopes of generalized snake posets}

As mentioned earlier, the volume of an order polytope $\mathcal{O}(P)$ is determined by the number of linear extensions of the poset $P$. 
The authors in \cite{vonBell+} study the volume of $\mathcal{O}(P(\bw))$ by considering the recursive structure of the poset of upper order ideals of $P(\bw)$.  
We recall some of their results here.

\begin{theorem}[Theorem 3.6 and Corollaries 3.7 \& 3.8, \cite{vonBell+}]
\label{thm:O(P(w))Volume}
For $n\geq0$, let $\bw=w_0 w_1\cdots w_n$ be a generalized snake word.
If $k\geq0$ is the largest index such that $w_k\neq w_n$, then the normalized volume $v_n$ of $\mathcal{O}(P(\bw))$ is given recursively by 
\[v_n = \mathrm{Cat}(n-k+1)v_{k}
    +\left(\mathrm{Cat}(n-k+2)-2\cdot \mathrm{Cat}(n-k+1)\right)v_{k-1}\]
with $v_{-1}=1$ and $v_0=2$, where $\mathrm{Cat}(m)=\frac{1}{m+1}\binom{2m}{m}$ is the  \emph{$m$-th Catalan number}.
In particular, 
\begin{itemize}
    \item for $S_n=P(\varepsilon LRLR\cdots)$ or $S_n=P(\varepsilon RLRL\cdots)$, 
the normalized volume of $\mathcal{O}(S_n)$ with $n\geq 0$ is given recursively by
$$ v_n = 2v_{n-1} + v_{n-2},$$
with $v_{-1}=1$ and $v_0 = 2$ (these are the Pell numbers), and
\item for $\mathcal{L}_n=P(\varepsilon RRRR\cdots)$ or $\mathcal{L}_n=P(\varepsilon LLLL\cdots)$, the normalized volume of $\mathcal{O}(\mathcal{L}_n)$ with $n\geq 0$ is given by 
$$v_n = \mathrm{Cat}(n+2).$$
\end{itemize}
\end{theorem}

The next result states that the normalized volume of an order polytope $\mathcal{O}(P(\bw))$ of a generalized snake poset is bounded above and below by the volume of the order polytope of the ladder poset and the snake poset, respectively.
Note that the dimension of $\mathcal{O}(P(\bw))$ is $d: = 2 n + 4$ for a snake word of length $n$.

\begin{theorem}[Theorem 3.10, \cite{vonBell+}]\label{thm:minmaxvolumes}
For any generalized snake word $\bw= w_0w_1\cdots w_n$ of length $n$,
$$\vol \mathcal{O}(S_n) \leq 
\vol \mathcal{O}(P(\bw)) \leq 
\vol \mathcal{O}(\mathcal{L}_n).$$
\end{theorem}

Recall that a lattice polytope is \emph{Gorenstein} if there exists a positive integer $k$ such that
$$
(k-1)P^\circ \cap \Z^n = \emptyset, |kP^\circ \cap \Z^n| =1, |tP^\circ \cap \Z^n| =|(t-k)P \cap \Z^n|
$$
for all integers $t>k$.

This is equivalent to the polytope having a symmetric $h^*$-polynomial \cite{BeckRobins}. 
Furthermore, the order polytope of a poset $P$ is Gorenstein if and only if $P$ is graded (i.e., all maximal
chains have the same length).
Observe that $P(\bw)$ is graded, thus we have that $\mathcal{O}(P(\bw))$ is Gorenstein and has a symmetric $h^*$-polynomial.

By also applying Ehrhart reciprocity and the above facts, we expect the root distribution of $L_{\mathcal{O}(P(\bw))}(t)$ to be symmetric, which motivates Theorem \ref{thm:roots} and the computation of its exact line of symmetry.

\begin{remark}\label{endswithR}
To determine the Ehrhart polynomial (or $h^*$-polynomial) of an order polytope $\mathcal{O}(P)$ one focuses on the poset structure rather than the natural number labeling of the elements in the poset. 
Hence, note that the Ehrhart polynomial of $\calO(P(\bw))$ corresponds with the Ehrhart polynomial of $\calO(P(\Bar{\bw}))$ where $\Bar{\bw}$ is a generalized snake word obtained by replacing all $R$'s and $L$'s in $\bw$ with $L$'s and $R$'s, respectively. Because of this, we only consider the generalized snake words that end with the letter $R$ (without loss of generality), unless explicitly stated otherwise.
\end{remark}

\begin{theorem} \label{thm:roots}
For any generalized snake word $\bw$ of length $n$, we have the following:
    \begin{enumerate}[topsep=1ex,itemsep=1ex,partopsep=1ex,parsep=1ex]
        \item $(t + 1)(t + 2)\cdots (t + n + 3)  \quad|\quad L(\bw; t)$, \label{thm:roots_item1}
        
        \item $L(\bw; t) = L(\bw; -n-4-t)$, and \label{thm:roots_item2}
        
        \item all integer roots of $L(\bw; t)$ lie in the interval $[- n - 4, 0]$. \label{thm:roots_item3}
    \end{enumerate}
\end{theorem}

\begin{proof}
Proving Theorem \ref{thm:roots}(\ref{thm:roots_item1}), is equivalent, by Ehrhart--Macdonald reciprocity, to showing that there are no interior integer points in the $t^{\text{th}}$ dilate of $\mathcal{O}(P(\bw))$ when $1 \leq t \leq n + 3$, for any generalized snake word $\bw$ of length $n$.

Towards that goal, let $\operatorname{rk}_i=\{2n+3-2i, 2n+4-2i\}$, for $1\leq i \leq n+1$, be the set of rank $i$ elements in $P(\bw)$ with $\operatorname{rk}_0 = \{2n + 3\}$ and $\operatorname{rk}_{n + 2}  = \{0\}$.
Consider, also, the associated order polytope $\mathcal{O}(P(\bw))$ that has dimension $d=2n+4$.
We will show that when $t > n + 3$,  $\mathcal{O}(P(\bw))^{\circ}  \cap \frac{1}{t} \Z^d \neq \varnothing$.
Assume that there is an integer point $q \in  \mathcal{O}(P(\bw))^{\circ}  \cap \frac{1}{t} \Z^d $.
Recall from \cite{StanleyTwoPosetPolytopes}, that a point $q \in \mathcal{O}(P(\bw))$ can be interpreted as a function  $P \to \R_{\geq 0}$ and then embedded to $\R^d$, so we use $q(i)$ to denote its value at $i \in P(\bw)$.
Notice that a point $q$ in $\mathcal{O}(P(\bw))$ is only on the boundary when $q(i) = 0$ or $1$, or when $q(i) = q(j)$ for some $i, j \in P(\bw)$. 
Since $q$ is in the interior, thus does not lie on the boundary of $\mathcal{O}(P(\bw))$, the following strict inequalities hold: 
\[q(0) < 1,\; \; q(2n +3 ) > 0,\; \text{ when } q(i) \neq q(j), \text{ for } i \in \operatorname{rk}_m,\; j \in \operatorname{rk}_{m + 1}.\]
Since $q \in \frac{1}{t}\Z^d$, $tq(2n + 3)$ is an integer, $tq(2n + 3) \geq 1$ as $tq(2n + 3) > 0$. 
    So, it follows that:
    \[tq(0) < t,\; tq(2n + 3) \geq 1,\; tq(i) \neq tq(j) \in \Z,\; i \in \operatorname{rk}_m,\; j \in \operatorname{rk}_{m + 1}.\]
    Thus, $t q(0) \geq n + 3 $, which entails that $ t > n + 3$.
    Hence,  when $1 \leq t \leq n + 3$, we have that $\mathcal{O}(P(\bw))^{\circ}  \cap \frac{1}{t} \Z^d = \varnothing$.
    
    Equivalently, when $1 \leq t \leq n + 3$, $L(\mathcal{O}(P(\bw))^{\circ}; t) = 0$. 
    Applying Ehrhart--Macdonald reciprocity to the aforementioned statement, when $ -1 \geq t \geq - n -3, t \in \N$, $L(\mathcal{O(\bw)}(P); t) = 0$, as desired. 

Next, we demonstrate Theorem \ref{thm:roots}(\ref{thm:roots_item2}) by providing a bijection.
By Ehrhart--Macdonald reciprocity, we may equivalently show that \[L(\mathcal{O}(P(\bw)); t) = L({\mathcal{O}(P(\bw))^{\circ}}; n + 4 + t).\] 
For any $q \in t \mathcal{O}(P(\bw)) \cap \Z^{d}$, define $T: t \mathcal{O}(P(\bw)) \cap \Z^{d} \rightarrow (n + 4 + t)\mathcal{O}(P)^{\circ} \cap \Z^{d}$ as follows:
     \[ T(q(2n + 3 -2i)) = \frac{1}{n + 4 + t}(tq(2n + 3 - 2i) + i + 1) \; \text{ and }\]
    \[T(q(2n + 4 -2i)) = \frac{1}{n + 4 + t}(tq(2n + 4 - 2i) + i + 1).\]
    Note that $T(q)(2n + 3 -2i) \neq T(q)(2n + 4 -2i)$, so $T$ is injective, since it also has an inverse, $T$ is bijective. 
    Furthermore, $L(t)$ and $L(- n -4 - t)$ agree for all positive integers, since they are finite-degree polynomials, they agree on $\R$. 

Finally, Theorem \ref{thm:roots}(\ref{thm:roots_item3}) is proven by combining \ref{thm:roots}(\ref{thm:roots_item1}) and \ref{thm:roots}(\ref{thm:roots_item2}).   
\end{proof}

\begin{example}
    The order polytope of the ladder, $\mathcal{O}(P(\epsilon RR))$, 
 is an $8$-dimensional lattice polytope with Ehrhart polynomial
    \[ L(\epsilon RR ; t)
    = \left(\frac{1}{2880}\right) \cdot(t+1) \cdot(t+5) \cdot(t+2)^2 \cdot(t+3)^2 \cdot(t+4)^2\]
    
    with volume $\frac{1}{2880}$ and $h^*$-polynomial
    \[h^*(\epsilon RR ; z) = 
    z^3 + 6 z^2 + 6 z + 1.\]
    
 The order polytope of the regular snake, $\mathcal{O}(P(\epsilon LR))$, is an $8$-dimensional convex polytope with Ehrhart polynomial
\[L(\epsilon LR; t)
    =\left(\frac{1}{3360}\right) \cdot(t+1) \cdot(t+2) \cdot(t+4) \cdot(t+5) \cdot(t+3)^2 \cdot\left(t^2+6 t+\frac{28}{3}\right)\]
    with volume $\frac{1}{3360}$ and $h^*$-polynomial
    \[h^*(\epsilon LR ; z) = 
    z^3 + 5 z^2 + 5 z + 1.\] 
    These computations were done with Sage\cite{Sage}.
\end{example}

\subsection{Notation}\label{subsec:notation}
\text{} 

 We conclude our preliminaries with a notation table for reference. 
Note that it includes notation for terms that will be defined in later sections.

\begin{table}[ht]
    \centering
    \caption{Notation Table}
    \begin{tabular}{|c|c|}
      \hline
        \textbf{Symbol} & \textbf{Meaning} \\ \hline
        $\bw$ & a generalized snake word \\ \hline
        $P(\bw)$ & the generalized snake poset corresponding to the generalized snake word $\bw$ \\ \hline
        $J(P(\bw))$ & the poset of upper order ideals of poset $P(\bw)$, ordered by reverse inclusion \\ \hline
        $\mathcal{O}(P(\bw))$ & the order polytope of poset $P(\bw)$ \\ \hline

        $\mathcal{P}^{\circ}$ & the interior of the polytope $\mathcal{P}$ \\ \hline
        $L(\calP; t)$ & the number of lattice points in $t\calP$ where $t \in \Z$ \\ \hline
        $h^*(\calP ; z)$ & the $h^*$-polynomial of a polytope $\calP$ \\ \hline
        $h^*(\bw ; z)$ & the $h^*$-polynomial of a polytope $\calO(P(\bw))$ (:= $h^*(\calO(P(\bw)); z)$) \\ \hline
        $\mathcal{C}_k(P)$ & the set of length $k$ chains in poset $P$. \\ \hline
        
        $\calD_k$ & the set of linear extensions $\pi$ of a given poset such that $\operatorname{asc}(\pi) = k$ \\ \hline
        $\calM$ & the set of all maximal chains of a given poset\\ \hline
        
        $\langle a_1, \dots, a_n \rangle$ & an upper order ideal generated by elements $a_1, \dots, a_n$ of a given poset\\ \hline

        $\bw[i:j]$ & the subword of $\bw = w_0 \dots w_n$ from the $i^{\text{th}}$ index to the $j^{\text{th}}$ \\ \hline
        $P(\bw)_{[i:j]}$ & the sub-poset of  $P(\bw)$ corresponding to subword $\bw[i:j]$ (Definition \ref{segment}) \\ \hline
        $J(P(\bw))_{[i:j]}$ & the sub-poset of  $J(P(\bw))$ corresponding to subword $\bw[i:j]$ (Definition \ref{segment}) \\ \hline        
    \end{tabular}
    \label{tab:notation}
\end{table}


\section{A combinatorial formula for the chain polynomial}
Given a triangulation $T$ of a $d$-polytope, let $f_k$ be the number of $k$-simplices of $T$, which are encoded in the following $h$-polynomial:
\[h_T(z):=\sum_{k=-1}^d f_kz^{k+1}(1-z)^{d-k}.\]
It is known that the $h^*$-polynomial of a lattice polytope is equal to the $h$-polynomial of any unimodular triangulation of the polytope \cite{BetkeMcMullen, StanleyDecompositions}. 
Hence, by considering the canonical triangulation $\mathcal{T}$ of $\mathcal{O}(P(\bw))$ as defined in \cite{StanleyTwoPosetPolytopes} and mentioned above in Subsection \ref{subsec:order_polytopes}, we can study the $h$-polynomial accompanying $\mathcal{T}$ to gain insight into the structure of $h^*(\mathcal{O}(P(\bw));z)$. 
Furthermore, by work of Stanley \cite[Section 5]{StanleyTwoPosetPolytopes} we know that the face numbers $f_k$ of $\mathcal{T}$ coincide with the the number of chains of a length $k$ in $J(P(\bw))$, denoted in what follows as $c_k$. 
Therefore, in order to determine the $h$-polynomial of $\mathcal{T}$ for $\mathcal{O}(P(\bw))$ (and hence $h^*(\bw; z)$), it suffices to determine a formula for the chain polynomial $\mathcal{C}(z)$ of $J(P(\bw))$.
Hence, we can recover $h^*(\bw; z)$ from $\mathcal{C}(z) = \sum_{k = 0}^d c_k z^{k + 1} $ by the following formula:
\[
h^*(\bw; z) =
\sum_{k = 0}^d c_k z^{k} (1 - z)^{d - k}.
\]

We derive $\mathcal{C}(z)$ via recurrence and encapsulate the formula by enumerating \emph{weighted valid paths}.
To start, we consider the embedding of the Hasse diagram of $J(P)$ onto $\Z^2$. 
Note that elements get larger going rightward or upward, following the poset order. 
For fixed coordinates, $i$ and $j$, we define $J_{i, j}$ as the subposet of $J(P)$ where we take the elements embedded in $(x, y)$ where $x \geq i$ and $y \geq j$, i.e., \[J_{i,j}:= \{(x, y): x \geq i \text{ and } y \geq j\} \bigcap J(P).\] 
See Figure \ref{fig: C_{i,j} example} for two examples of $J_{i,j}$.
Next, denote the corresponding chain polynomial of $J_{i,j}$ by $\mathcal{C}_{i,j}(z)$.
Furthermore, for those elements in $J(P)$ where $y=j$, we can consider a \emph{chain polynomial at height} $j$ and point out that those polynomials can be expressed as linear combination of chain polynomials at height $j+1$, as the next lemma shows. 

\begin{lemma}\label{lemma:chain_polynomial}
Consider the set of elements of $J(P)$ at height $j$: $\{(i_{1,j}, j), (i_{2,j}, j), \cdots, (i_{m_j,j}, j)\}$.
The chain polynomial $\mathcal{C}_{i,j}(z)$ is given recursively by
\begin{equation}\label{eq: chain recur formula}
    \mathcal{C}_{i,j}(z) = \sum\limits_{k = 1}^{m_{j + 1}} z^{\delta_{i, i_{k,j + 1}}} (1 + z)^{i_{k,j + 1} - i} \mathcal{C}_{i_{k,j + 1}, j + 1}(z)
\end{equation}
when $i \geq j$ and
\[\delta_{i, i_{k,j + 1}} = \begin{cases}
        1, & i \neq i_{k,j + 1}\\
        0, & i = i_{k,j + 1}.
    \end{cases}
    \]
When $i < j$, $\mathcal{C}_{i,j}(z)$ is given by 
\begin{equation}\label{eq: chain recur formula2}
    \mathcal{C}_{i,j}(z) = (1 + z)^{j - i} \mathcal{C}_{j,j}(z).
\end{equation}
\end{lemma}

We may think of Lemma \ref{lemma:chain_polynomial} as \emph{expanding chain polynomials height by height}.
In particular, under our notation:
\[\mathcal{C}_{-1,0}(z) = \mathcal{C}(z).\]
To avoid a disarray of notation, we prove Lemma \ref{lemma:chain_polynomial} in the specific case where $P$ is a ladder.
We note that in general, the proof structure works similarly. 
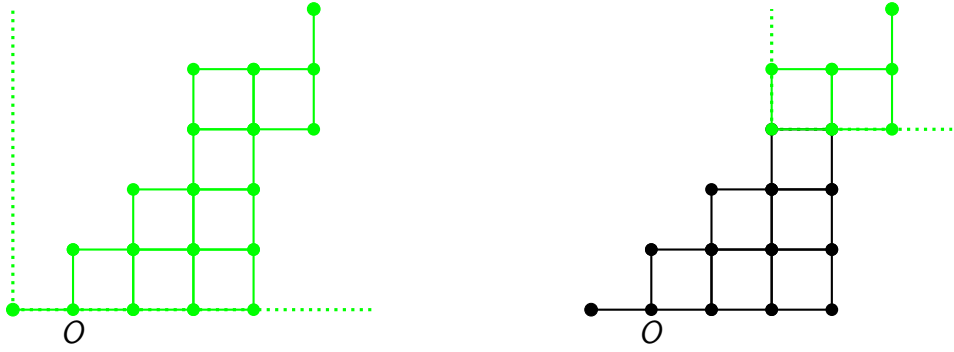
\begin{figure}[htbp]
    \centering
    \begin{tikzpicture}[scale=0.8]
    \def\xmin{-1}
    \def\xmax{4}
    \def\ymin{0}
    \def\ymax{5}
    
    \draw (0, 0) node[below] {$O$};   
    
    \draw[green, dotted, very thick] (-1,0) -- (5,0);
    \draw[green, dotted, very thick] (-1,0) -- (-1,5);
    
    \draw[green,thick] (0,0) rectangle (1,1);
    \foreach \x/\y in {0/0, 1/0, 0/1, 1/1}
        \fill[green] (\x,\y) circle (3pt);
    \draw[green,thick] (1,0) rectangle (2,1);
    \foreach \x/\y in {1/0, 2/0, 1/1, 2/1}
        \fill[green] (\x,\y) circle (3pt);
    \draw[green,thick] (2,0) rectangle (3,1);
    \foreach \x/\y in {2/0, 3/0, 0/1, 3/1}
        \fill[green] (\x,\y) circle (3pt);
    \draw[green,thick] (1,1) rectangle (2,2);
    \foreach \x/\y in {1/1, 2/1, 1/2, 2/2}
        \fill[green] (\x,\y) circle (3pt);
    \draw[green,thick] (2,1) rectangle (3,2);
    \foreach \x/\y in {2/1, 3/1, 2/2, 3/2}
        \fill[green] (\x,\y) circle (3pt);
    \draw[green,thick] (2,2) rectangle (3,3);
    \foreach \x/\y in {2/2, 3/2, 2/3, 3/3}
        \fill[green] (\x,\y) circle (3pt);
    \draw[green,thick] (2,3) rectangle (3,4);
    \foreach \x/\y in {2/3, 3/3, 2/4, 3/4}
        \fill[green] (\x,\y) circle (3pt);
    \draw[green,thick] (3,3) rectangle (4,4);
    \foreach \x/\y in {3/3, 4/3, 3/4, 4/4}
        \fill[green] (\x,\y) circle (3pt);
    
    \draw[green, thick] (-1, 0) -- (0 , 0);
    \draw[green, thick] (4, 4) -- (4 , 5);
    \filldraw[green] (-1,0) circle (3pt);
    \filldraw[green] (4, 5) circle (3pt);
\end{tikzpicture}
\hspace{2.5 cm}
    \begin{tikzpicture}[scale=0.8]
    \def\xmin{-1}
    \def\xmax{4}
    \def\ymin{0}
    \def\ymax{5}
    
    \draw (0, 0) node[below] {$O$};   
    
    \draw[green, dotted, very thick] (2,3) -- (2,5);
    \draw[green, dotted, very thick] (2,3) -- (5,3);
    
    \draw[black,thick] (0,0) rectangle (1,1);
    \foreach \x/\y in {0/0, 1/0, 0/1, 1/1}
        \fill[black] (\x,\y) circle (3pt);
    \draw[black,thick] (1,0) rectangle (2,1);
    \foreach \x/\y in {1/0, 2/0, 1/1, 2/1}
        \fill[black] (\x,\y) circle (3pt);
    \draw[black,thick] (2,0) rectangle (3,1);
    \foreach \x/\y in {2/0, 3/0, 0/1, 3/1}
        \fill[black] (\x,\y) circle (3pt);
    \draw[black,thick] (1,1) rectangle (2,2);
    \foreach \x/\y in {1/1, 2/1, 1/2, 2/2}
        \fill[black] (\x,\y) circle (3pt);
    \draw[black,thick] (2,1) rectangle (3,2);
    \foreach \x/\y in {2/1, 3/1, 2/2, 3/2}
        \fill[black] (\x,\y) circle (3pt);
    \draw[black,thick] (2,2) rectangle (3,3);
    \foreach \x/\y in {2/2, 3/2, 2/3, 3/3}
        \fill[black] (\x,\y) circle (3pt);
    \draw[green,thick] (2,3) rectangle (3,4);
    \foreach \x/\y in {2/3, 3/3, 2/4, 3/4}
        \fill[green] (\x,\y) circle (3pt);
    \draw[green,thick] (3,3) rectangle (4,4);
    \foreach \x/\y in {3/3, 4/3, 3/4, 4/4}
        \fill[green] (\x,\y) circle (3pt);
    
    \draw[black, thick] (-1, 0) -- (0 , 0);
    \draw[green, thick] (4, 4) -- (4 , 5);
    \filldraw (-1,0) circle (3pt);
    \filldraw[green] (4, 5) circle (3pt);
\end{tikzpicture}
    \caption{ Subposets $J_{-1,0}$ (left) and $J_{2,3}$ (right), colored in green}
    \label{fig: C_{i,j} example}
\end{figure}

\begin{proof}[Proof of Lemma 3.1 for the ladder case]
Consider $i \geq j$ and note that at height $j + 1$, we have \[\{(j , j +1), (j + 1, j + 1), \cdots, (n + 1, j + 1)\}.\]
Let $\mathcal{C}_k(J_{i,j})$ be the set of length $k$ chains in $J_{i,j}$ and take $|\mathcal{C}_k(J_{i,j})|$ to be the coefficient of $z^k$ in $\mathcal{C}_{i,j}(z)$.
Consider any chain $\mathcal{C}$ in $\mathcal{C}_k(J_{i,j})$ and let $x$ be the largest index such that $(x, j) \in \mathcal{C}$.
When there is not a $(x, j) \in \mathcal{C}$, let $x = 0$ and we shall see that $x$ determines the rest of $\mathcal{C}$.
Denote $\mathcal{C} \backslash (\mathcal{C} \cap \{(i , j)\;|\; i = j - 1, j , \cdots, n + 1\})$ by $\hat{\mathcal{C}}$.
We have the following:
    \begin{enumerate}
        \item[1.] If $x = 0$, namely $\mathcal{C}$ does not have an element at height $j$, then $\hat{\mathcal{C}} \in \mathcal{C}_k(J_{i,j + 1})$.
        \item[2.] If $x = i$, then $\hat{\mathcal{C}} \in \mathcal{C}_{k - 1}(J_{i,j + 1})$.
        \item[3.] If $x = m, i + 1\leq m \leq n + 1$, suppose the number of points in $\mathcal{C}$ of height $j$ is $q$, where $1 \leq q \leq \min\{m - i + 1, k\}$, i.e.,
        \[
        |\mathcal{C} \cap \{(i , j)| i = j - 1, j , \cdots, n + 1\}| = q.
        \]
        To determine $\mathcal{C} \cap \{(i , j)\;|\; i = j - 1, j , \dots, n + 1\}$, we have $\binom{m - i }{q - 1}$ ways to choose $q - 1$ elements from $\{(d , j)\;|\; d = i, i + 1 , \dots, m - 1\}$, besides $(m, j + 1)$. 
        Note that regardless of the choice of these $q - 1$ points, we have
         $\hat{\mathcal{C}} \in \mathcal{C}_{k - q}(J_{m,j + 1})$. 
         Thus, this $x = m$ case contributes in total \[\sum\limits_{q = 1}^{\min\{m - i + 1, k\}}\binom{m - i}{q - 1}|\mathcal{C}_{k - q}(J_{m,j + 1})| = \sum\limits_{q = 1}^{\min\{m - i + 1, k\}}\binom{m - i}{q - 1}[z^{k - q}]V_{m, j + 1}(z)\] to our enumeration of $|\mathcal{C}_k(J_{i,j})|$.
    \end{enumerate}
    Therefore, under the assumption that $i \geq j$, we have: 
    \[
    [z^k]\mathcal{C}_{i,j}(z) =  [z^k]\mathcal{C}_{i,j+1}(z) + [z^{k - 1}]\mathcal{C}_{i,j+ 1}(z) + \sum\limits_{m = i + 1}^{n + 1}\sum\limits_{q = 1}^{\min\{m - i + 1, k\}}\binom{m - i}{q - 1}[z^{k - q}]\mathcal{C}_{m, j + 1}(z).
    \]
    Because $[z^k]\mathcal{C}_{i,j+1}(z) + [z^{k - 1}]\mathcal{C}_{i,j+ 1}(z) = [z^k](1 + z)\mathcal{C}_{i,j+1}(z)$, we have the following identity:
    \begin{equation} \label{eq: eq2 in lemma 3.1}
       [z^k]\mathcal{C}_{i,j}(z) = [z^k](1 + z)\mathcal{C}_{i,j+1}(z) + \sum\limits_{w = 2}^{n + 2 - i}\sum\limits_{q = 1}^{\min\{w, k\}}\binom{w - 1}{q - 1}[z^{k - q}]\mathcal{C}_{w + i - 1, j + 1}(z).
    \end{equation}

    Notice that the desired identity (\ref{eq: chain recur formula}) for the ladder case is:
    \[
    \mathcal{C}_{i, j}(z) = (1 + z)\mathcal{C}_{i, j + 1}(z) + \sum_{k = 2}^{n + 1} z(1 + z)^{k - 1} \mathcal{C}_{i+k-1, j + 1}(z).
    \]
   It suffices to prove:
    \begin{equation} \label{eq: eq3 in lemma 3.1}
        [z^k] \mathcal{C}_{i, j}(z) = [z^k](1 + z)\mathcal{C}_{i,j+1}(z) + \sum_{w = 2}^{n + 2 - i} [z^k] z(1 + z)^{w - 1} \mathcal{C}_{w + i -1, j + 1}(z).
    \end{equation}
    Now, substituting $[z^k] z(1 + z)^{w - 1} \mathcal{C}_{w + i -1, j + 1}(z) = \sum\limits_{q = 1}^{\min\{w, k\}}\binom{w - 1}{q - 1}[z^{k - q}]\mathcal{C}_{w + i - 1, j + 1}(z)$ into (\ref{eq: eq2 in lemma 3.1}), we get exactly the desired equality (\ref{eq: eq3 in lemma 3.1}), thus concludes the proof when $i \geq j$.
    
    When $i = j - 1$, the desired identity (\ref{eq: chain recur formula2}) is:
    \[
    \mathcal{C}_{j-1, j}(z) = (1 + z).
    \]
    Similarly, for any $\mathcal{C} \in \mathcal{C}_k(J_{j - 1, j})$, we have the following two cases:
    \begin{itemize}
        \item[1.] If $(j-1, j) \in \mathcal{C}$, then $\mathcal{C} \setminus \{(j-1, j)\} \in \mathcal{C}_{k - 1}(J_{j, j})$.
        \item[2.] If $(j-1, j) \notin \mathcal{C}$, then $\mathcal{C} \setminus \{(j-1, j)\} \in \mathcal{C}_{k}(J_{j, j})$.
    \end{itemize}
    Thus, $[z^k] \mathcal{C}_{j-1, j}(z) = [z^k] \mathcal{C}_{j, j}(z) + [z^{k - 1}] \mathcal{C}_{j, j}(z) = [z^k] (1 + z) \mathcal{C}_{j, j}(z)$, giving us identity (\ref{eq: chain recur formula2}) for the $i = j - 1$ case.
\end{proof}

This allows us to expand $\mathcal{C}_{-1,0}(z) = \mathcal{C}(z)$ by height as:
\begin{align*}
    \mathcal{C}_{-1, 0}(z) &=  A_{0,0}(z)\mathcal{C}_{0, 0}(z)\\
    &= A_{i_{1,1},1}(z)\mathcal{C}_{i_{1,1},1 }(z) + A_{i_{2,1},1}(z)\mathcal{C}_{i_{2,1},1 }(z) + \cdots + A_{i_{m_1, 1},1}(z)\mathcal{C}_{i_{m_1, 1},1 }(z)\\
    &= A_{n + 1, n +2}(z)\mathcal{C}_{n + 1, n+ 2}(z).
\end{align*}
In fact, to obtain the chain polynomial, it suffices to determine the polynomial $A_{n + 1, n + 2}(z)$.
One can either view the polynomials $A_{i,j}$ as placeholders or alternatively, we can compute them via Lemma \ref{lemma:chain_polynomial}.

Note that the third equality holds because we can apply Lemma \ref{lemma:chain_polynomial} to $\mathcal{C}_{i_{k,1},1 }(z)$, for $1\leq k \leq m_1$, to obtain the expansion at height $2$:
\[
 \mathcal{C}_{-1, 0}(z) = A_{i_{1,2},2}(z)\mathcal{C}_{i_{1,2},2 }(z) + A_{i_{2, 2},2}(z)\mathcal{C}_{i_{2, 2},2}(z) + \cdots + A_{i_{m_2, 2},2}(z)\mathcal{C}_{i_{m_2, 2},2 }(z).
\]
Then continue to apply Lemma \ref{lemma:chain_polynomial} to get the expansion at height $n + 2$:
\[
\mathcal{C}(z) = \mathcal{C}_{-1, 0}(z) = A_{n + 1, n +2}(z)\mathcal{C}_{n + 1, n+ 2}(z).
\]

By repeatedly applying Lemma \ref{lemma:chain_polynomial} we could obtain $A_{n + 1, n +2}(z)$, but we introduce a combinatorial interpretation of $A_{n + 1, n +2}(z)$, next.
In order to state the combinatorial formula of $A_{n + 1, n +2}(z)$ concisely, we define our combinatorial objects, namely: \emph{valid paths}, the \emph{weight of a path}, and the \emph{total weights} from one point to another point.

\begin{definition}\hfill
    \begin{enumerate}
        \item[1.] Let $e_{i,j}$ denote a step from  $(i,j)$ to $(i + 1, j)$, where $(i, j + 1) \notin J(P)$ and let 
        $h_{i, j}^k$ denote a step from $(i, j)$ to $(i + k, j + 1)$ when $(i + k, j + 1) \in J(P)$.
        A \textit{valid path} $p$ is a sequence of steps on $J(P)$ consisting of $e_{i,j}$ and $h_{i, j}^k$.
        
        \item[2.] Take $e_{i,j}$ to have weight $1 + z$, $h_{i,j}^0$ to have weight $1 + z$, and $h_{i, j}^k, k > 0$ to have weight $z(1 + z)^k$.
        The \emph{weight of a valid path} is the product of the weights of all its steps.
        \item[3.] Take $W\{(i,j), (p,q)\}$ to be the sum of weights of all possible valid paths from $(i, j)$ to $(p,q)$.
    \end{enumerate}
\end{definition}

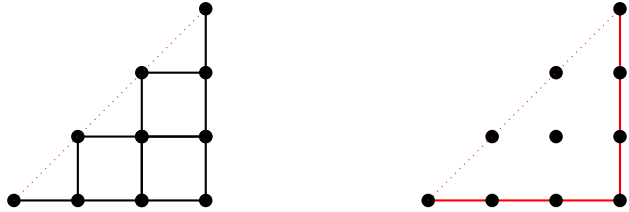
\begin{figure*}[hbt!]
    \centering
    \begin{tikzpicture}[scale = 0.85]
    \def\xmin{-4}
    \def\xmax{4}
    \def\ymin{0}
    \def\ymax{5}
    \draw[black,thick] (0,0) rectangle (1,1);
    \foreach \x/\y in {0/0, 1/0, 0/1, 1/1}
        \fill (\x,\y) circle (3pt);
    \draw[black,thick] (1,0) rectangle (2,1);
    \foreach \x/\y in {1/0, 2/0, 1/1, 2/1}
        \fill (\x,\y) circle (3pt);
    \draw[black,thick] (1,1) rectangle (2,2);
    \foreach \x/\y in {1/1, 2/1, 1/2, 2/2}
        \fill (\x,\y) circle (3pt);

    \draw[black, thick] (-1, 0) -- (0 ,0 );
    \draw[black, thick] (2, 2) -- (2, 3);
    \fill (-1, 0) circle (3pt);
    \fill (2, 3) circle (3pt);
    \draw[dotted, red, thin] (-1, 0) -- (2,3);
\end{tikzpicture}
\hspace{2.5cm}
\begin{tikzpicture}[scale = 0.85]
        \def\xmin{-4}
        \def\xmax{4}
        \def\ymin{0}
        \def\ymax{5}
        
        \draw[red, thick] (-1, 0) -- (0 ,0 ) -- (1 , 0) -- (2,0) --
        (2,1) -- (2,2) -- (2, 3);

        \foreach \x/\y in {-1/0, 0/0, 1/0, 2/0, 2/1, 0/1, 1/1, 1/2, 2/2, 2/3}
        \fill[black] (\x,\y) circle (3pt);
        \draw[dotted, red, thin] (-1, 0) -- (2,3);
\end{tikzpicture}
    \caption{$J(P(\epsilon R))$ (left) and a path that is not valid (right, in red).}
    \label{fig: invalid path example}
\end{figure*}

With these definitions and structure in hand, we now present a combinatorial formula for the chain polynomial of a generalized snake.

\begin{theorem}\label{thm:chain_coefficients}
The polynomial $A_{n+1, n+2}(z)$ is given by 
\[A_{n+1, n+2}(z) = W\{(-1,0), (n+1, n+2)\}.\]
Thus, the chain polynomial of any generalized snake is given by \[\mathcal{C}(z) = W\{(-1,0), (n+1, n+2)\} \cdot (1 + z).\] 
\end{theorem}

\begin{proof}[Proof of Theorem 3.3 for the ladder case]
We consider the ladder case since it makes the notation much clearer, but the proof technique holds for any generalized snake.
In height $i$, we have $A_{i, i}, A_{i + 1, i}, \cdots, A_{n + 2, i}$, by applying Lemma \ref{lemma:chain_polynomial} we have: \[\mathcal{C}_{-1,0}(z) = (1 + z) \mathcal{C}_{0, 0} (z), \text{ and }\]
\[\mathcal{C}_{0,0} (z) = (1 + z) V_{0,1} (z) + z(1 + z) V_{1,1} (z) + \cdots + z(1 + z)^{n + 1}\mathcal{C}_{n + 1, 1} (z).\]
Combining these two identities, we get:
    \[\mathcal{C}_{-1, 0} (z) = W\{(-1,0), (0, 1)\}\mathcal{C}_{0,1} (z) + \cdots + W\{(-1,0), (n+1, 1)\}\mathcal{C}_{n + 1, 1} (z).\]
Thus, $A_{i,1} = W\{(-1,0), (i, 1)\}$, keep using the lemma to expand $\mathcal{C}_{i,1}$ to the next height and we obtain that $A_{i, j} = W\{(-1,0), (i, j)\}$.  
\end{proof}

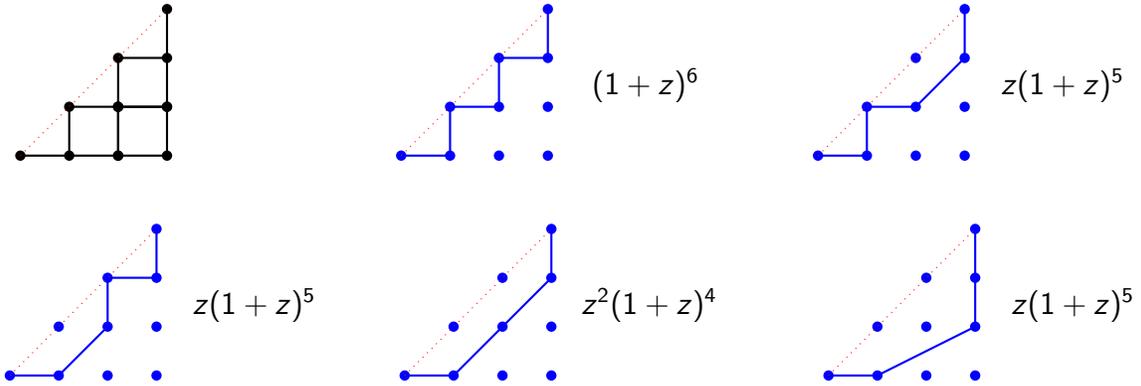
\begin{figure*}[hbt!]
    \centering
    \begin{tikzpicture}[scale = 0.65]
    \def\xmin{-4}
    \def\xmax{4}
    \def\ymin{0}
    \def\ymax{5}
    \draw[black,thick] (0,0) rectangle (1,1);
    \foreach \x/\y in {0/0, 1/0, 0/1, 1/1}
        \fill (\x,\y) circle (3pt);
    \draw[black,thick] (1,0) rectangle (2,1);
    \foreach \x/\y in {1/0, 2/0, 1/1, 2/1}
        \fill (\x,\y) circle (3pt);
    \draw[black,thick] (1,1) rectangle (2,2);
    \foreach \x/\y in {1/1, 2/1, 1/2, 2/2}
        \fill (\x,\y) circle (3pt);

    \draw[black, thick] (-1, 0) -- (0 ,0 );
    \draw[black, thick] (2, 2) -- (2, 3);
    \fill (-1, 0) circle (3pt);
    \fill (2, 3) circle (3pt);
    \draw[dotted, red, thin] (-1, 0) -- (2,3);
    \end{tikzpicture}
    \hspace{2.7cm}
    \begin{tikzpicture}[scale = 0.65]
        \def\xmin{-4}
        \def\xmax{4}
        \def\ymin{0}
        \def\ymax{5}
        
        \draw[blue, thick] (-1, 0) -- (0 ,0 ) -- (0 , 1) -- (1,1) --
        (1,2) -- (2,2) -- (2, 3);
        \node[anchor=north] at (4,2) {$(1 + z)^6$};
        \foreach \x/\y in {-1/0, 0/0, 1/0, 2/0, 2/1, 0/1, 1/1, 1/2, 2/2, 2/3}
        \fill[blue] (\x,\y) circle (3pt);
        \draw[dotted, red, thin] (-1, 0) -- (2,3);
    \end{tikzpicture}
    \hspace{1.1cm}
    \begin{tikzpicture}[scale = 0.65]
        \def\xmin{-4}
        \def\xmax{4}
        \def\ymin{0}
        \def\ymax{5}
        \draw[blue, thick] (-1, 0) -- (0 ,0 ) -- (0 , 1) -- (1,1) --
        (2, 2) -- (2, 3);
        \node[anchor=north] at (4,2) {$z (1 + z)^5$};
        \foreach \x/\y in {-1/0, 0/0, 1/0, 2/0, 2/1, 0/1, 1/1, 1/2, 2/2, 2/3}
        \fill[blue] (\x,\y) circle (3pt);
        \draw[dotted, red, thin] (-1, 0) -- (2,3);
    \end{tikzpicture}
    \vspace{0.8cm}

    \begin{tikzpicture}[scale = 0.65]
        \def\xmin{-4}
        \def\xmax{4}
        \def\ymin{0}
        \def\ymax{5}
        \draw[blue, thick] (-1, 0) -- (0 ,0 ) -- (1,1) --
        (1,2) -- (2,2) -- (2, 3);
        \node[anchor=north] at (4,2) {$z (1 + z)^5$};
        \foreach \x/\y in {-1/0, 0/0, 1/0, 2/0, 2/1, 0/1, 1/1, 1/2, 2/2, 2/3}
        \fill[blue] (\x,\y) circle (3pt);
        \draw[dotted, red, thin] (-1, 0) -- (2,3);
    \end{tikzpicture}
    \hspace{0.7cm}
    \begin{tikzpicture}[scale = 0.65]
        \def\xmin{-4}
        \def\xmax{4}
        \def\ymin{0}
        \def\ymax{5}
        \draw[blue, thick] (-1, 0) -- (0 ,0 ) -- (2, 2) -- (2, 3);
        \node[anchor=north] at (4,2) {$z^2 (1 + z)^4$};
        \foreach \x/\y in {-1/0, 0/0, 1/0, 2/0, 2/1, 0/1, 1/1, 1/2, 2/2, 2/3}
        \fill[blue] (\x,\y) circle (3pt);
        \draw[dotted, red, thin] (-1, 0) -- (2,3);
    \end{tikzpicture}
    \hspace{1cm}
    \begin{tikzpicture}[scale = 0.65]
        \def\xmin{-4}
        \def\xmax{4}
        \def\ymin{0}
        \def\ymax{5}
        \draw[blue, thick] (-1, 0) -- (0 ,0 ) -- (2,1) -- (2,2) -- (2, 3);
        \node[anchor=north] at (4,2) {$z(1 + z)^5$};
        \foreach \x/\y in {-1/0, 0/0, 1/0, 2/0, 2/1, 0/1, 1/1, 1/2, 2/2, 2/3}
        \fill[blue] (\x,\y) circle (3pt);
        \draw[dotted, red, thin] (-1, 0) -- (2,3);
    \end{tikzpicture}
    \caption{All maximal valid paths in $J(P(\epsilon R))$}
    \label{fig: weighted path example}
\end{figure*}

\begin{example} \label{example: using section 3 formula}
 Figure \ref{fig: weighted path example} depicts all the maximal valid paths of $J(P(\epsilon R))$, which can be used to compute the associated chain polynomial:
    \begin{align*}
        \calC(J(P(\epsilon R)), z) &= (1 + z)^6 + 3 z (1 + z)^5 + z^2 (1 + z)^4 \\
        &= 5 z^6 + 25 z^5 + 51 z^4 + 54 z^3 + 31 z^2 + 9 z + 1
    \end{align*}
    Then we can obtain $h^*(\mathcal{O}(P(\epsilon R));z)$, via the relation between $h^*$-polynomial of the polytope and the $h$-polynomial of its unimodular triangulation:
    \begin{align*}
        h^*(\mathcal{O}(P(\epsilon R)); z) &=  (1 -z )^6 + 9 z (1 - z)^5 + 31 z^2 (1 - z)^4 \\
        &+ 54 z^3 (1 - z)^3 + 51 z^4 (1 - z)^2 + 25 z^5 (1 - z) + 5 z^6\\
        &= z^2 + 3 z + 1.
    \end{align*}
\end{example}
\label{sec:chain_polynomial}


\section{ \texorpdfstring{$h^*$-}-polynomials via lattice-path enumeration}
\label{sec:lattice-path}
In what follows, we obtain the closed formulas for the $h^*$-polynomial of the ladder and the regular snake, a recurrence formula for the $h^*$-polynomial of any generalized snake, and the coefficient-wise $h^*$-inequality with the ladder and the regular snake having the greatest and smallest coefficient respectively, all via lattice-path enumeration.

\subsection{The ladder and the regular snake}
Here, we present two closed formulas for the $h^*$-polynomials of the order polytopes of ladders and regular snakes. We will see that they exhibit nice coefficients: the Narayana numbers and the Delannoy numbers, respectively.

\bigskip 

\begin{definition}\hfill
    \begin{enumerate}
        \item[1.] The Narayana numbers \( N(n, k) \) are defined by the following formula:
\[
N(n, k) = \frac{1}{n} \binom{n}{k} \binom{n}{k-1}.
\]
The Narayana numbers \( N(n, k) \) are known to count the number of Dyck paths of length \( 2n \) with exactly \( k \) peaks (OEIS \cite{OEIS}, A001263).
    \item[2.] The Delannoy numbers \( D(m, n) \) are given by the following formula
    \[
    D(m,n ) = \sum_{k = 0}^{\min (m,n)} \binom{m}{k} \binom{n}{k} 2^k\] and are known to count the number of paths from \((0, 0)\) to \((m, n)\) in the $\mathbb{Z}^2$ grid with additional diagonally up-right edges, where moves are allowed to the right, up, or diagonally up-right (OEIS \cite{OEIS}, A008288).
    \end{enumerate}
\end{definition}

Next, we introduce a lemma that is fundamental to our method of computing the $h^*$-polynomials.
The $\Omega$-Eulerian polynomial $E(\Omega, t)$ of a graded poset $\Omega$ is defined in \cite{MR2146855} as follows:
for $\pi:=\pi_1\cdots \pi_n \in \mathcal{S}_n$,
    \[ E(\Omega;z):= \sum\limits_{\pi \in \mathcal{L}(\Omega)} z^{\operatorname{des}(\pi)}.
    \]
where $\operatorname{des}(\pi) = |\{i \in [n - 1] \;|\; \pi_{i + 1} < \pi_i\}|$ and `$<$' is the natural order on integers.

\begin{lemma} [Theorem 3.7, \cite{MR2146855}] \label{lemma: Eulerian polynomial}
    Let $P$ be a naturally labeled and graded poset. 
    Since $L(\mathcal{O}(P); t)$ is equal to the number of order reversing maps $\rho : P \to \{0, 1, \ldots, t\}$, we have
\[\sum_{t \geq 0} L(\mathcal{O}(P); t) z^t = \frac{E(P; z)}{(1 - z)^{\dim(\mathcal{O}(P))+1}}.\]
In particular, $h^*(\mathcal{O}(P); z) = E(P; z)$.
\end{lemma}
\begin{remark}
    Lemma \ref{lemma: Eulerian polynomial} comes from the relationship between $L(\mathcal{O}(P); t)$ and the order polynomial of $P$. 
    Even though $P(\bw)$ is not naturally labeled in its original definition, we can take its dual poset which has the same Ehrhart polynomial and $h^*$-polynomial as $P(\bw)$, then the lemma applies. 
    Another way to obtain this lemma is by combining Theorem 3.5 in \cite{MR2146855} and the fact that $\mathcal{O}(P(\bw))$ is a compressed integral polytope (Theorem 1.1 in \cite{MR1838375}).
\end{remark}

Note that Lemma \ref{lemma: Eulerian polynomial} applies to \textit{naturally labeled posets}.
Since $P(\bw)$ is reversely naturally labeled, we replace $\operatorname{des}(\pi)$ with $\operatorname{asc}(\pi) := |\{i \in [n - 1] \; |\; \pi_{i + 1} > \pi_i\}|$.
Thus, it suffices to enumerate $\mathcal{D}_k := \{\pi \in \mathcal{L}(P)\, \mid \, \operatorname{asc}(\pi) = k\}$. 
Let $\mathcal{M}$ be the set of all maximal chains in $J(P(\bw))$. 
Since linear extensions of $P(\bw)$ correspond to $\mathcal{M}$ in $J(P(\bw))$, we color edges of the Hasse diagram $J(P(\bw))$ following the next construction and define a subset $\mathcal{M}_k$ of $\mathcal{M}$ based on that coloring.

\bigskip

\begin{construction}[\textit{Red L Turns}]\label{construction:red_L_turns}

We define the \textit{red L turn} and specify the set $\mathcal{M}_k$ of maximal chains at the end.

Let $\mathcal{C}$ be any maximal chain from $(-1, 0)$ to $(2n + 1, 2n + 2)$ in $J(P)$.
Let the linear extension corresponding to $\mathcal{C}$ be denoted by $e$, where each move of $\mathcal{C}$ adds a new number to $e$.
\bigskip
\paragraph{Step 1: Edge Coloring}
Starting from $(-1, 0)$, we color the edges of $\mathcal{C}$:
\begin{itemize}
    \item Note that $\mathcal{C}$ moves from $(i, j)$ to $(i+1, j+1)$ through either:
    \begin{enumerate}
        \item $(i, j) \to (i+1, j) \to (i+1, j+1)$, or
        \item $(i, j) \to (i, j+1) \to (i+1, j+1)$.
    \end{enumerate}
    \item Each move adds two numbers, $a$ and $b$, to $e$. 
    Depending on the order ($ab$ or $ba$, where $a \neq b$), one of these paths increases the number of ascents in $\mathcal{C}$.
    \item \textbf{Red L Turn}: The path that increases the number of ascents is called a \textit{red L turn}; otherwise it is considered a normal turn.
\end{itemize}

For each turn in $\mathcal{C}$, color the turn red if it increases the number of ascents. All other segments remain uncolored.
Refer to Figure \ref{fig: construction example} for an example illustrating red L turns.

\bigskip
\begin{figure}[htbp] 
    \centering
    \begin{tikzpicture}[scale=0.85]
    \def\xmin{-1}
    \def\xmax{3}
    \def\ymin{0}
    \def\ymax{4}

    \draw[black,thick] (0,0) rectangle (1,1);
    \foreach \x/\y in {0/0, 1/0, 0/1, 1/1}
        \fill (\x,\y) circle (3pt);
    \draw[black,thick] (1,0) rectangle (2,1);
    \foreach \x/\y in {1/0, 2/0, 1/1, 2/1}
        \fill (\x,\y) circle (3pt);
    \draw[black,thick] (2,2) rectangle (3,3);
    \foreach \x/\y in {2/2, 3/2, 2/3, 3/3}
        \fill (\x,\y) circle (3pt);
    \draw[black,thick] (1,1) rectangle (2,2);
    \foreach \x/\y in {1/1, 2/1, 1/2, 2/2}
        \fill (\x,\y) circle (3pt);
    \draw[black,thick] (1,2) rectangle (2,3);
    \foreach \x/\y in {1/2, 2/2, 1/3, 2/3}
        \fill (\x,\y) circle (3pt);

    \draw[black, thick] (-1, 0) -- (0 , 0);
    \draw[black, thick] (3, 3) -- (3, 4);
    \filldraw (-1,0) circle (3pt);
    \filldraw (3, 4) circle (3pt);

    \draw[red, thick] (0.1, 0.1) -- (0.9, 0.1);
    \draw[red, thick] (0.9, 0.1) -- (0.9, 0.9);
    
    \draw[red, thick] (1.1, 0.1) -- (1.9, 0.1);
     \draw[red, thick] (1.9, 0.1) -- (1.9, 0.9);
     
     \draw[red, thick] (1.1, 1.1) -- (1.1, 1.9);
     \draw[red, thick] (1.1, 1.9) -- (1.9, 1.9);
     
     \draw[red, thick] (1.1, 2.1) -- (1.1, 2.9);
     \draw[red, thick] (1.1, 2.9) -- (1.9, 2.9);
     
     \draw[red, thick] (2.1, 2.1) -- (2.9, 2.1);
     \draw[red, thick] (2.9, 2.1) -- (2.9, 2.9);

     \draw (- 1, 0) node[left] {$(- 1, 0)$};
     \draw (0, 1) node[left] {$(0, 1)$};

\end{tikzpicture}
\hspace{2 cm}
\begin{tikzpicture}[scale=0.85]
    \def\xmin{-1}
    \def\xmax{3}
    \def\ymin{0}
    \def\ymax{4}

    \draw[black,thick] (0,0) rectangle (1,1);
    \foreach \x/\y in {0/0, 1/0, 0/1, 1/1}
        \fill (\x,\y) circle (3pt);
        
    \draw[black,thick] (1,0) rectangle (2,1);
    \foreach \x/\y in {1/0, 2/0, 1/1, 2/1}
        \fill (\x,\y) circle (3pt);

    \draw[black,thick] (1,1) rectangle (2,2);
    \foreach \x/\y in {1/1, 2/1, 1/2, 2/2}
        \fill (\x,\y) circle (3pt);
        
    \draw[black,thick] (2,0) rectangle (3,1);
    \foreach \x/\y in {2/0, 3/0, 2/1, 3/1}
        \fill (\x,\y) circle (3pt);
        
    \draw[black,thick] (2,1) rectangle (3,2);
    \foreach \x/\y in {2/1, 3/1, 2/2, 3/2}
        \fill (\x,\y) circle (3pt);
        
    \draw[black,thick] (2,2) rectangle (3,3);
    \foreach \x/\y in {2/2, 3/2, 2/3, 3/3}
        \fill (\x,\y) circle (3pt);

    \draw[black, thick] (-1, 0) -- (0 , 0);
    \draw[black, thick] (3, 3) -- (3, 4);
    \filldraw (-1,0) circle (3pt);
    \filldraw (3, 4) circle (3pt);

    \draw[red, thick] (0.1, 0.1) -- (0.9, 0.1);
    \draw[red, thick] (0.9, 0.1) -- (0.9, 0.9);
    
    \draw[red, thick] (1.1, 0.1) -- (1.9, 0.1);
     \draw[red, thick] (1.9, 0.1) -- (1.9, 0.9);

    \draw[red, thick] (1.1, 1.1) -- (1.9, 1.1);
     \draw[red, thick] (1.9, 1.1) -- (1.9, 1.9);

    \draw[red, thick] (2.1, 0.1) -- (2.9, 0.1);
     \draw[red, thick] (2.9, 0.1) -- (2.9, 0.9);

    \draw[red, thick] (2.1, 1.1) -- (2.9, 1.1);
     \draw[red, thick] (2.9, 1.1) -- (2.9, 1.9);

    \draw[red, thick] (2.1, 2.1) -- (2.9, 2.1);
     \draw[red, thick] (2.9, 2.1) -- (2.9, 2.9);

     \draw (- 1, 0) node[left] {$(- 1, 0)$};
     \draw (0, 1) node[left] {$(0, 1)$};
     
\end{tikzpicture}
    \caption{Red L turns in $J(P(\epsilon LR))$ (left) and $J(P(\epsilon RR))$ (right)}
    \label{fig: construction example}
\end{figure}
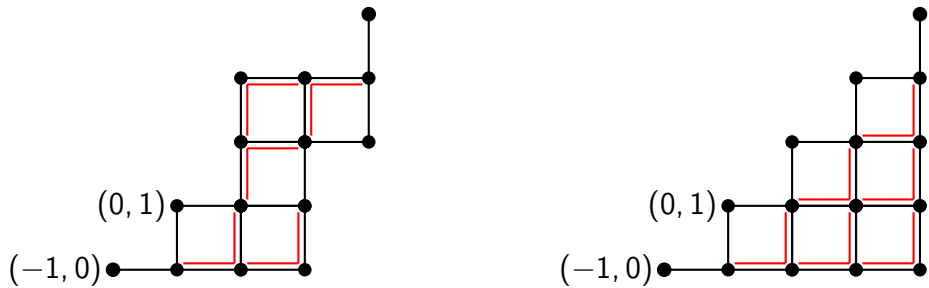


\paragraph{Step 2: Complete Red L Turns}
A maximal chain $\mathcal{C}$ is said to \textbf{contain} a red L turn if and only if both steps of the red L turn is in $\mathcal{C}$.
\begin{itemize}
    \item the red L turn is $(i, j) \to (i+1, j) \to (i+1, j+1)$ or $(i, j) \to (i, j + 1) \to (i+1, j+1)$, then $\mathcal{C}$ must include both steps.
    \item  If $\mathcal{C}$ contains only one step of the turn (e.g., $(i, j) \to (i+1, j)$ or $(i+1, j) \to (i+1, j+1)$, but not both), then it does not contain the red L turn.
\end{itemize}

\paragraph{Step 3: Full Coloring of $J(P(\bw))$}
To color the entire poset $J(P(\bw))$, repeat this process for all possible maximal chains. Note that every square face in $J(P)$ contains exactly one red L turn. 
\bigskip

\paragraph{Step 4: Defining $\mathcal{M}_k$}
Finally, we define $\mathcal{M}_k$ as:
\[
\mathcal{M}_k := \text{the set of maximal chains containing exactly $k$ red L turns}.
\]
 \hfill $\diamondsuit$
\end{construction}

\bigskip
We go on to prove that there is a bijection between $\mathcal{M}_k$, the set of maximal paths with $k$ red L turns, and $\mathcal{D}_k$, hence they have the same cardinality.

\begin{definition} \label{segment}
    Let $\bw[i : j]$ be any substring of $\bw = w_0w_1 \dots w_n$.
    We take $P(\bw)_{[i:j]}$ to be the subposet of $P(\bw)$ which:
    \begin{itemize}
        \item consists of the elements forming the square faces that are in $P(w_0w_1 \dots w_j)$ but not in $P(w_0w_1 \dots w_{i-1})$, and
        \item maintains the same partial order.
    \end{itemize} 
    We analogously define the subposet of $J(P(\bw))$ and denote it by $J(P(\bw))_{[i:j]}$.
\end{definition}
See Figure \ref{fig:subposet} for an example of the definition.

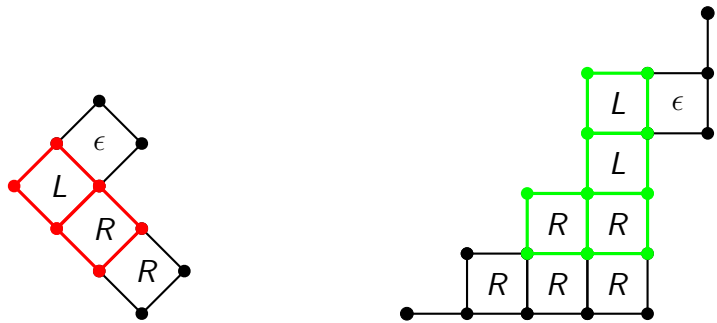
\begin{figure*}[hbt!]
    \centering
    \begin{tikzpicture}[scale=0.8, rotate=-45]
        \draw[black, thick] (-1,1) rectangle (-2,2);
        \foreach \x/\y in {-1/1, -2/1, -1/2, -2/2} {
            \fill (\x,\y) circle (3pt);
        }
        \draw (-1.7, 1.3) node[right] {$\epsilon$};
        
        \draw[black, thick] (0,0) rectangle (1,1);
        \foreach \x/\y in {0/0, 1/0, 0/1, 1/1} {
            \fill (\x,\y) circle (3pt);
        }
        \draw (0.3, 0.3) node[right] {$R$};

        \draw[red, very thick] (0,0) rectangle (-1,1);
        \foreach \x/\y in {0/0, -1/0, 0/1, -1/1} {
            \fill[red] (\x,\y) circle (3pt);
        }
        \draw (-0.7, 0.3) node[right] {$R$};
   
        \draw[red, very thick] (-2,1) rectangle (-1,0);
        \foreach \x/\y in {-2/1, -1/1, -2/0, -1/0} {
            \fill[red] (\x,\y) circle (3pt);
        }
        \draw (-1.7, 0.3) node[right] {$L$};
        
\end{tikzpicture}
\hspace{2.5 cm}
\begin{tikzpicture}[scale=0.8]
    \def\xmin{-1}
    \def\xmax{4}
    \def\ymin{0}
    \def\ymax{5}   
    
    \draw[black,thick] (0,0) rectangle (1,1);
    \foreach \x/\y in {0/0, 1/0, 0/1, 1/1}
        \fill[black] (\x,\y) circle (3pt);
    \node at (0.5, 0.5) {$R$};

    \draw[black,thick] (1,0) rectangle (2,1);
    \foreach \x/\y in {1/0, 2/0, 1/1, 2/1}
        \fill[black] (\x,\y) circle (3pt);
    \node at (1.5, 0.5) {$R$};

    \draw[black,thick] (2,0) rectangle (3,1);
    \foreach \x/\y in {2/0, 3/0, 0/1, 3/1}
        \fill[black] (\x,\y) circle (3pt);
    \node at (2.5, 0.5) {$R$};

    \draw[green, very thick] (1,1) rectangle (2,2);
    \foreach \x/\y in {1/1, 2/1, 1/2, 2/2}
        \fill[green] (\x,\y) circle (3pt);
    \node at (1.5, 1.5) {$R$};

    \draw[green, very thick] (2,1) rectangle (3,2);
    \foreach \x/\y in {2/1, 3/1, 2/2, 3/2}
        \fill[green] (\x,\y) circle (3pt);
    \node at (2.5, 1.5) {$R$};

    \draw[green, very thick] (2,2) rectangle (3,3);
    \foreach \x/\y in {2/2, 3/2, 2/3, 3/3}
        \fill[green] (\x,\y) circle (3pt);
    \node at (2.5, 2.5) {$L$};

    \draw[black,thick] (3,3) rectangle (4,4);
    \foreach \x/\y in {3/3, 4/3, 3/4, 4/4}
        \fill[black] (\x,\y) circle (3pt);
    \node at (3.5, 3.5) {$\epsilon$};
    
    \draw[green, very thick] (2,3) rectangle (3,4);
    \foreach \x/\y in {2/3, 3/3, 2/4, 3/4}
        \fill[green] (\x,\y) circle (3pt);
    \node at (2.5, 3.5) {$L$};    
    
    \draw[black, thick] (-1, 0) -- (0 , 0);
    \draw[black, thick] (4, 4) -- (4 , 5);
    \filldraw (-1,0) circle (3pt);
    \filldraw (4, 5) circle (3pt);
\end{tikzpicture}
    \caption{Subposets $P(\bw)_{[1:2]}$ colored by red (left), $J(P(\bw))_{[1:2]}$ colored by green (right) for $\bw = \epsilon LRR$.}
    \label{fig:subposet}
\end{figure*}

\begin{lemma} \label{lemma: red L turn}
    The set of all length $2$ paths in the Hasse diagram of $J(P(\bw))$ that contributes one ascent is exactly the set of all red $L$-turns in $J(P(\bw))$. 
\end{lemma}

\begin{proof}
    To avoid confusion, denote $\precdot_{\mathcal{L}}$ to be the cover relation in terms of the total order relation defined by a linear extension $\mathcal{L}$ and $<$ to be the natural ordering of integers.

    First, note that every pair of elements $(a_1, a_2)$ that satisfies both $a_1 \precdot_{\mathcal{L}} a_2$ and $a_1 < a_2$ for some linear extension $\mathcal{L}$, is included entirely in either one of the following two types of subposets of $P(\bw)$:
    \begin{enumerate}
        \item $P(\bw)_{[i:j]}$, where $\bw[i : j]$ is any substring of $\bw$ consisting of only one type of letter. \label{case:1}
        \item $P(\bw)_{[m:m+1]}$, where $\bw[m:m+1] \in \{RL, LR\}$.\label{case:2}
    \end{enumerate}
    Thus, it is sufficient to prove that the lemma holds, for each cases.
    
    We first prove the lemma for elements in a subposet of type (\ref{case:1}).
    Let $\bw[i : j]$ be any substring of $\bw$ consisting of only one type of letter, $L$.
    Take $\mathcal{L}$ to be any linear extension of $P(\bw)$.
    By doing a case analysis on all the pairs of elements in $P(\bw)_{[i:j]}$, say $(a_1, a_2)$, such that satisfies both $a_1 \precdot_{\mathcal{L}} a_2$ and $a_1 < a_2$, observe that either 
    \begin{enumerate}
    \item[(i)]  no such pair exists, or
    \item[(ii)] there exist such pairs and all of them are of the form $(2k+1, 2k+2a+2)$, where $2k+1, 2k+2a+2 \in P(\bw)_{[i:j]}$ and $k, a \in \Z_{\geq 0}$.
    \end{enumerate}

    From the preceding facts, we have that the set of all length 2 paths in the Hasse diagram of $J(P(\bw))_{[i:j]}$ that contributes one ascent is of the form 
    \[\langle 2k+1, 2k+2a+2 \rangle \rightarrow \langle 2k-1, 2k+2a+2 \rangle \rightarrow \langle 2k-1, 2k + 2a \rangle\]
    
    for $\langle 2k+1, 2k+2a+2 \rangle, \langle 2k-1, 2k+2a+2 \rangle, \langle 2k-1, 2k+2a \rangle \in J(P(\bw))_{[i:j]}$ (We define $\langle -1, x \rangle := \langle x \rangle$).
    Here, note that \[\langle 2k+1, 2k+2a+2 \rangle \rightarrow \langle 2k-1, 2k+2a+2 \rangle \rightarrow \langle 2k-1, 2k+2a \rangle\]
    is a $L$ turn in the Hasse diagram of $J(P(\bw))_{[i:j]}$, allowing us to conclude that every length $2$ paths in $J(P(\bw))_{[i:j]}$ that contributes to an ascent is a $L$ turn, specifically a red $L$ turn as defined in Construction \ref{construction:red_L_turns}.

    The case where $\bw[i : j]$ is any maximal substring of $\bw$ consisting of only one type of letter $R$ is analogous.
    Thus, we have shown that the lemma holds for the elements in a subposet of type (\ref{case:1}).

    Next we prove the lemma for elements in a subposet of type (\ref{case:2}).
    Observe that every three elements $\ba_1, \ba_2, \ba_3$ in $J(P(\bw))_{[m:m+1]}$ where $\bw[m:m+1] \in \{RL, LR\}$ that satisfies both $(\ba_2 \setminus \ba_1 )\precdot_{\mathcal{L}} (\ba_3 \setminus \ba_2)$ and $(\ba_2 \setminus \ba_1) < (\ba_3 \setminus \ba_2)$, is completely contained in either $J(P(\bw))_{[m:m]}$ or $J(P(\bw))_{[m+1:m+1]}$.
    Thus, every pair of elements in $P(\bw)_{[m:m+1]}$, say $(a_1, a_2)$, that satisfies both $a_1 \precdot_{\mathcal{L}} a_2$ and $a_1 < a_2$ is in $P(\bw)_{[m:m]}$ or $P(\bw)_{[m+1:m+1]}$.
    Hence, as \[\bw[m:m], \bw[m+1:m+1] \in \{R, L\},\] every such pair $(a_1, a_2)$ is also a pair of elements in a subposet of type ($\ref{case:1}$), moreover such pair of elements that correspond to an ascent, form also the only two elements in a specific rank of $P(\bw)$.
    Therefore, as the first case was already shown, the proof is complete.
\end{proof}

Lemma \ref{lemma: red L turn} establishes a bijection between $\mathcal{M}_k$ and $\mathcal{D}_k$.

\begin{theorem} \label{thm: bijection}
    Let $e: \mathcal{M} \to \mathcal{D}$ be the bijection mapping maximal chains in $J(P(\bw))$ to their corresponding linear extensions of $P(\bw)$, then for any $k$ such that $0 \leq  k \leq \textit{deg}(h^*)$.
We have that
    $e|_{\mathcal{M}_k}$ is a bijection between $\mathcal{D}_k$ and $\mathcal{M}_k$.
\end{theorem}
Specifically, the $\mathcal{D}_k$ and $\mathcal{M}_k$ for a regular snake word are in bijection. 
We provide a more explicit proof of the regular snake case.

\begin{proof}[Proof of Theorem 4.7 for the regular snake case]

For consistency, take a regular snake word $\bw$ that ends with $R$. 
It suffices to prove that if a length 2 path does not form a red L turn, it will not contribute to, the set of ascents in the linear extension the path corresponds to.

 Label the squares in $J(P(\bw))$ with $S_{i, j}$ and $C_i$ as in Figure \ref{fig: snakeproof}, where they can be thought of as side square faces and central square faces.
\begin{enumerate}
    \item If $\mathcal{C}$ does not contain a red L turn from any $S_{i, i+1}$, then linear extension $e(\mathcal{C})$, has the form $(2n+1)(a_1b_1) (a_2b_2) \cdots (a_n b_n) 0$, where $a_ib_i$ are the two numbers contributed by a certain square.
    Then $\min(a_i, b_i) =  1 + \max(a_{i + 1}, b_{i + 1})$, namely $b_i > a_{i + 1}$, thus $b_i a_{i + 1}$ does not contribute the set of ascents in the linear extension. 
    Thus, the claim holds in this case.
    \item Now suppose $\mathcal{C}$ does contain a red L turn coming from $S_{i, i+1}$. 
If $S_{i, i+1}$ is below the central chain (colored in blue, vertices on the line $y = x$), for example, the red L turn $(1,0) \to (2,0) \to (2,1)$ in $S_{1,2}$ in Figure \ref{fig: snakeproof}. 
It has this path segment (in coordinates): 
$$
(i , i) \in C_i \to (i + 1, i) \to (i + 2, i) \to (i +2, i + 1) \to (i + 2, i + 2) \in C_{i + 1}
$$
If $S_{i, i+1}$ is above the central chain, it has this path segment: $$
(i , i) \in C_i \to (i , i + 1) \to (i, i + 1) \to (i + 1, i + 2) \to (i + 2, i + 2) \in C_{i + 1}
$$

this segment contributes $x(ab)y$ to the linear extension $e(\mathcal{C})$, where $x,y$ are from the central squares, $ab$ is from $S_{i, i+1}$). We determine $x(ab)y$ explicitly:

the central point $(i, i)$ is the upper order ideal $\{0, 1, 2, \cdots, 2n - 2i\}$, so 
\begin{align*}
    \{x, a, b, y\} &= \{0, 1, 2, \cdots, 2n - 2i\} \setminus \{0, 1, 2, \cdots, 2n - 2i - 4\} \\ 
    &= \{2n - 2i, 2n - 2i - 1, 2n - 2i -2 , 2n - 2i - 3\}
\end{align*}
If $S_{i, i+1}$ is below the central chain:
$$ x= 2n - 2i - 1, \text{\;\;\;} a = 2n - 2i - 2, \text{\;\;\;} b = 2n - 2i, \text{\;\;\;} y = 2n - 2i - 3.
$$
If $S_{i, i+1}$ is above the central chain:
$$ x= 2n - 2i - 1, \text{\;\;\;} a = 2n - 2i - 3, \text{\;\;\;} b = 2n - 2i, \text{\;\;\;} y = 2n - 2i - 2.
$$
both cases come from the construction of $J(P)$.
So in either case we can see that $a > x, b > y$, namely $ax$ and $by$ will not contribute to the set of ascents.
Thus, the claim holds for the regular snake. \qedhere
\end{enumerate}
\end{proof}

Next, we obtain the $h^*$-polynomial of the regular snake.

\begin{theorem}\label{thm: h*_snake}
The $h^*$-polynomial of $\mathcal{O}(P(\bw))$ when $\bw$ is a regular snake word of length $n - 1, n \geq 1$, i.e., $\bw= \epsilon LRLR\dots$ or $\epsilon RLRL \dots$, is of the form
        \[
            h^*(\bw; z) = \sum\limits_{i = 0}^n D(n - i, i)z^i, 
        \]
where $D(i,j)$ denotes the Delannoy number.   

\end{theorem}

\begin{proof}
By Lemma \ref{lemma: Eulerian polynomial}, we have that
\[h^*(\bw; z) = \sum\limits_{i = 0}^n |\mathcal{D}_i| z^i.\]
Since $|\mathcal{D}_i| = |\mathcal{M}_i|$ by Theorem \ref{thm: bijection}, it suffices to prove 
\[|\mathcal{M}_i| = D(n - i, i).
\]
We do so by constructing a bijection since $D(n - i, i)$ itself counts a set of paths in a different lattice.

We have a chain of central boxes $C_j$ as in the left Hasse diagram of Figure \ref{fig: snakeproof}, and some $S_{i, i + 1}$ between $C_i$ and $C_{i + 1}$. 
Again, for consistency, we take a regular snake word $\bw$ that ends with $R$.
If $S_{i, i + 1}$ is below the chain of central boxes, $(i, j) \to (i + 1, j) \to (i+1, j + 1)$ is the red L turn in $S_{i, i + 1}$; if $S_{i, i + 1}$ is above the chain of central boxes, $(i, j) \to (i, j + 1) \to (i+1, j + 1)$ is the red L turn in $S_{i, i + 1}$.

From $(i, i)$, a maximal chain can travel to $(i + 1, i + 1)$ via either a red $L$ turn or a normal turn, or it can skip $(i + 1, i + 1)$ and travel directly to $(i + 2, i + 2)$ by using points not in the central squares $C_{i + 1}$ and $C_{i + 2}$, since there is only one such point (either $(i + 2, i)$ or $(i, i + 2)$) between $(i, i)$ and $(i + 2 , i + 2)$, there is only one way of performing a \textit{skip} from $(i,i)$ to $(i + 2, i + 2)$.

Now let $G$ be a $\mathbb{Z}^2$ lattice with northeast edges added between $(i,j)$ and $(i + 1, j + 1)$, for all $i, j \in \mathbb{N}$.
Take any $ \mathcal{C} \in \mathcal{M}_k$, $\mathcal{C}$ consists of a sequence of red L turns, normal turns, and skips. 
A red $L$ turn in $\mathcal{C}$ corresponds to moving $(0,1)$ in $G$, a normal turn correspond to moving $(1,0)$, and a skip correspond to moving $(1,1)$, this way we have a map from $\mathcal{M}_k$ to the set $G_k$, defined as the set of paths starting from $(0, 0)$ and ending at $(n-k,k)$ in $G$, with steps $(1, 0)$, $(0, 1)$, and $(1, 1)$. Notice that $|G_k| = D(n-k, k)$ by the definition of Delannoy numbers.

We call this mapping $T: \mathcal{M}_k \rightarrow G_K$ (we will prove later that $T(\mathcal{C}) $ is indeed in $G_k$). 
\bigskip
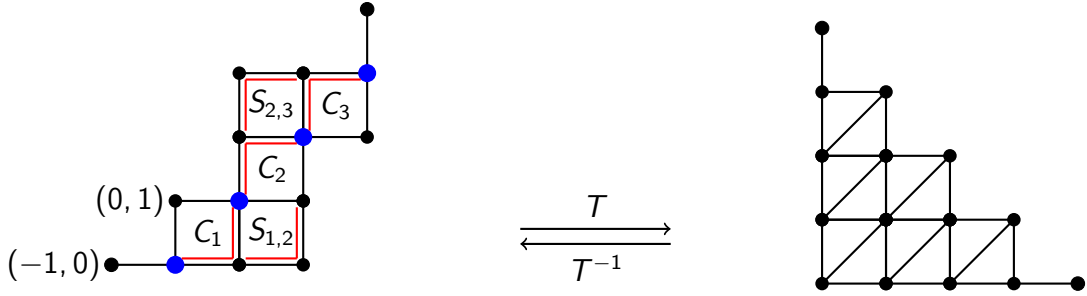
\begin{figure}[htbp]
    \centering
    \begin{tikzpicture}[scale=0.85]
    \def\xmin{-1}
    \def\xmax{3}
    \def\ymin{0}
    \def\ymax{4}

    \draw[black,thick] (0,0) rectangle (1,1);
    \foreach \x/\y in {0/0, 1/0, 0/1, 1/1}
        \fill (\x,\y) circle (3pt);
    \draw[black,thick] (1,0) rectangle (2,1);
    \foreach \x/\y in {1/0, 2/0, 1/1, 2/1}
        \fill (\x,\y) circle (3pt);
    \draw[black,thick] (2,2) rectangle (3,3);
    \foreach \x/\y in {2/2, 3/2, 2/3, 3/3}
        \fill (\x,\y) circle (3pt);
    \draw[black,thick] (1,1) rectangle (2,2);
    \foreach \x/\y in {1/1, 2/1, 1/2, 2/2}
        \fill (\x,\y) circle (3pt);
    \draw[black,thick] (1,2) rectangle (2,3);
    \foreach \x/\y in {1/2, 2/2, 1/3, 2/3}
        \fill (\x,\y) circle (3pt);

    \draw[black, thick] (-1, 0) -- (0 , 0);
    \draw[black, thick] (3, 3) -- (3, 4);
    \filldraw (-1,0) circle (3pt);
    \filldraw (3, 4) circle (3pt);
    \foreach \x/\y/\label in {1/1, 2/2, 3/3, 0/0}
        \fill[blue] (\x,\y) circle (4pt);

    \draw[red, thick] (0.1, 0.1) -- (0.9, 0.1);
    \draw[red, thick] (0.9, 0.1) -- (0.9, 0.9);
    
    \draw[red, thick] (1.1, 0.1) -- (1.9, 0.1);
     \draw[red, thick] (1.9, 0.1) -- (1.9, 0.9);
     
     \draw[red, thick] (1.1, 1.1) -- (1.1, 1.9);
     \draw[red, thick] (1.1, 1.9) -- (1.9, 1.9);
     
     \draw[red, thick] (1.1, 2.1) -- (1.1, 2.9);
     \draw[red, thick] (1.1, 2.9) -- (1.9, 2.9);
     
     \draw[red, thick] (2.1, 2.1) -- (2.9, 2.1);
     \draw[red, thick] (2.9, 2.1) -- (2.9, 2.9);

     \draw (- 1, 0) node[left] {$(- 1, 0)$};
     \draw (0, 1) node[left] {$(0, 1)$};
     \draw (0.5, 0.5) node[] {$C_1$};
     \draw (1.5, 0.5) node[] {$S_{1, 2}$};
      \draw (1.5, 1.5) node[] {$C_2$};
       \draw (2.5, 2.5) node[] {$C_3$};
       \draw (1.5, 2.5) node[] {$S_{2, 3}$};
    
\end{tikzpicture}
    \hspace{0.8 cm}
    \qquad 
\begin{tikzpicture}
    \def\xmin{-1}
    \def\xmax{4}
    \def\ymin{0}
    \def\ymax{4}
    \draw[->, thick] (1,4) -- (3,4) node[midway, above] {$T$};
    \draw[<-, thick] (1,3.8) -- (3,3.8) node[midway, below] {$T^{-1}$};
\end{tikzpicture}
\qquad %
\hspace{0.8 cm}
    \begin{tikzpicture}[scale=0.85]
    \def\xmin{-1}
    \def\xmax{3}
    \def\ymin{0}
    \def\ymax{4}

    \draw[black,thick] (0,0) rectangle (1,1);
    \foreach \x/\y in {0/0, 1/0, 0/1, 1/1}
        \fill (\x,\y) circle (3pt);
    \draw[black,thick] (1,0) rectangle (2,1);
    \foreach \x/\y in {1/0, 2/0, 1/1, 2/1}
        \fill (\x,\y) circle (3pt);
    \draw[black,thick] (2,0) rectangle (3,1);
    \foreach \x/\y in {2/0, 3/0, 2/1, 3/1}
        \fill (\x,\y) circle (3pt);

    \draw[black,thick] (0,2) rectangle (1,3);
    \foreach \x/\y in {0/2, 1/2, 0/3, 1/3}
        \fill (\x,\y) circle (3pt);
    \draw[black,thick] (0,1) rectangle (1,2);
    \foreach \x/\y in {0/1, 1/1, 0/2, 1/2}
        \fill (\x,\y) circle (3pt);
    \draw[black,thick] (1,1) rectangle (2,2);
    \foreach \x/\y in {1/1, 2/1, 1/2, 2/2}
        \fill (\x,\y) circle (3pt);

    \draw[black, thick] (3, 0) -- (4, 0);
    \draw[black, thick] (0, 3) -- (0, 4);
    \draw[black, thick] (0, 0) -- (1, 1);
    \draw[black, thick] (1, 0) -- (2, 1);
    \draw[black, thick] (2, 0) -- (3, 1);
    \draw[black, thick] (0, 1) -- (1, 2);
    \draw[black, thick] (0, 2) -- (1, 3);
    \draw[black, thick] (1, 1) -- (2, 2);

    \filldraw (4, 0) circle (3pt);
    \filldraw (0, 4) circle (3pt);

\end{tikzpicture}
    \caption{Red L turns, blue central points in $J(P(\epsilon LR))$, and $G$ (on the right)}
    \label{fig: snakeproof}
\end{figure}

\noindent{$T$ is a \textit{Bijection:}}\hfill
\\
\indent We first show that $T(\mathcal{C})$ is in $G_k$. 
Recall that $\mathcal{C}$ only consists of the three basic steps: a normal turn, red L turn, and skip. 
Since a skip contains a red L turn, so either a red $L$ turn or a skip add $1$ to the $y$ coordinate of $T(\mathcal{C})$. 
Since $\mathcal{C} \in \mathcal{M}_k$, it has $k$ red $L$ turns and skips combined, thus $T(\mathcal{C})$ ends at $y = k$ in $G$.
    
 As for the $x$ coordinate of $T(\mathcal{C})$ in $G$, let $a,b,c$ be the number of red $L$ turns, normal turns, and skips, respectively.
 Note that the total number of edges $\mathcal{C}$ uses from $0$ to $n$ is always $2n$, thus $2a + 2b + 4c = 2n$ and $a + c = k$, so we deduce $b + c = n - k$.
 Further notice that since either a skip or normal turn contributes $1$ to the $x$ coordinate, $T(\mathcal{C})_x = b + c = n -k$, thus $T(\mathcal{C}) \in G_k$. 

Next, we verify that the inverse map is well-defined, i.e., whenever there is a $(1, 1)$ in $T(\mathcal{C})$, the corresponding maximal chain can take a skip without leaving $J(P(\bw))$. 

Take any $s \in G_k$, suppose the number of steps $s$ takes is $n - k$, where $k\geq 0$ is the number of $(1,1)$ step, we can assume $k \geq 1$ since the other case maps back to a valid $\mathcal{C}$ with zero skips. 
Suppose that $s$ takes its last $(1,1)$ step at the $i^{\text{th}}$ step, $i \leq n - k$, after completing the first $i-1$ steps, corresponding $\mathcal{C}$ halts at $2(k-1) + i - 1 -(k -1 ) = (i + k -2)^{\text{th}}$ central point, since $i + k -2 \leq n -2 $, so $\mathcal{C}$ can still take a skip at the $(i + k -2)^{\text{th}}$ central point. 
Hence, we have inductively shown that $\mathcal{C} = T^{-1}(s)$ is a valid maximal chain which is in $\mathcal{M}_k$. Thus, $T$ is surjective. 

Since $T^{-1}: G_k \rightarrow \mathcal{M}_k$ is well-defined by reversing the correspondence between basic steps in $J(P(\bw))$ and $G$, the same analysis applies to $T^{-1}$, so $T$ is also injective. 
Thus, the bijection is established.
\end{proof}

Next, we derive the $h^*$-polynomial of the ladder poset using red L turns.

\begin{theorem}\label{thm:h*_ladder}
The $h^*$-polynomial of $\mathcal{O}(P(\bw))$ when $\bw$ is a ladder of length $n$, i.e., $\bw= \epsilon LLLL\dots$ or $\epsilon RRRR \dots$, is of the form
        \[
            h^*(\bw; z) = \sum\limits_{i = 0}^{n + 1} N(n + 2, i + 1) z^i
        \]
where $N(x,y)$ denotes the Narayana number.   
\end{theorem}

\begin{proof}
Let $\mathcal{N}_{n + 2}(z) = \sum\limits_{i = 0}^{n + 1} N(n + 2, i + 1) z^i$ denote the $(n + 2)^{\text{th}}$ Narayana polynomial, where $\mathcal{N}_0(z) = \mathcal{N}_1(z) = 1$.
A well-known non-linear recurrence relation of Narayana polynomials is given as follows for $n \geq 1$:
\begin{align*}
    \mathcal{N}_{n+2}(z) &= (1+z)\mathcal{N}_{n+1}(z) + z \sum_{k=1}^{n} \mathcal{N}_{k}(z)\mathcal{N}_{n-k+1}(z)\\
    &= \mathcal{N}_{n+1}(z) + z \sum_{k=0}^{n} \mathcal{N}_{k}(z)\mathcal{N}_{n-k+1}(z).
\end{align*}

We prove a similar recurrence relation for $h^*(\bw_n;z)$, where $\bw_n$ is a ladder of length $n$. The recurrence relation we will prove is the following:
\begin{equation}\label{narayanarec}
    h^*(\bw_n;z) = \mathcal{N}_{n+1}(z) + z \left( \sum_{k=0}^{n-1} \mathcal{N}_{k}(z) h^*(\bw_{n-k-1};z) + \mathcal{N}_{n}(z).\right)
\end{equation}

As $h^*(\bw_0;z) = z+1 = \mathcal{N}_2(z)$ and $h^*(\bw_1;z) = z^2+3z+1 = \mathcal{N}_3(z)$, we can deduce that $h^*(\bw_n;z) = \mathcal{N}_{n+2}(z)$ for all $n \geq 0$ by induction, since they share the same recurrence relation.

To prove Equation (\ref{narayanarec}), we consider the embedding of $J(P(\bw))$ onto $\Z^2$ (as depicted in the right picture of Figure \ref{fig: construction example}) where, without loss of generality, $\bw$ consists of the letter $L$'s and is of length $n$.
Note that a maximal chain of $J(P(\bw))$ is a lattice path from $(-1, 0)$ to $(n+1, n+2)$ that stays below the line $y = x + 1$. 
Let the set of all such lattice paths be denoted as $\mathcal{LP}$.

First, consider lattice paths that does not pass through any points that lie on the line $y = x + 1$ except for the points $(-1, 0)$ and $(n+1, n+2)$.
Let the set of all such lattice paths be denoted as $\mathcal{LP}^{(1)}$.
Observe that the ordinary generating function for the number of lattice paths in $\mathcal{LP}^{(1)}$ that passes through exactly $i$ red $L$-turns is $\mathcal{N}_{n+1}(z)$.

Next, consider the set of all lattice paths $\mathcal{LP}^{(2)}$ such that passes through at least one point that lies on the line $y = x + 1$ other than the points $(-1, 0)$ and $(n+1, n+2)$.
Let $\mathcal{LP}^{(2)}_k := \{l \in \mathcal{LP}^{(2)} \text{ } | \text{ }  \min_{(x, x+1) \in l}\{x : x \neq -1, x \neq n+1\} = k \}$.
Then, observe that the ordinary generating function for the number of lattice paths in $\mathcal{LP}^{(2)}_k$ that passes through exactly $i$ red $L$-turns is $z\mathcal{N}_{k}(z) h^*(\bw_{n-k-1};z)$ for $0 \leq k \leq n-1$, $z\mathcal{N}_n(z)$ for $k = n$, and $0$ for $k = n+1$.
Thus, we can deduce that the ordinary generating function for the number of lattice paths in $\mathcal{LP}^{(2)}$ that passes through exactly $i$ red $L$-turns is

\[z \left( \sum_{k=0}^{n-1} \mathcal{N}_{k}(z) h^*(\bw_{n-k-1};z) + \mathcal{N}_{n}(z)\right).\]

Therefore, the generating function for the number of lattice paths in $\mathcal{LP}$ that passes through exactly $i$ red $L$-turns is the right hand side of Equation (\ref{narayanarec}).
By Theorem \ref{thm: bijection}, we obtain Equation (\ref{narayanarec}).
The theorem follows since the following equalities holds:
\[
    \mathcal{N}_{n+2}(z) = z^{n+2} \sum\limits_{i = 0}^{n + 1} N(n + 2, i + 1) z^{-i} = \sum\limits_{i = 0}^{n + 1} N(n + 2, i + 1) z^i.
\]

\end{proof}

We now have that the $h^*$-polynomial obtained from Theorem \ref{thm:h*_ladder} aligns with $h^*$-polynomial using the combinatorial formula for chains in Theorem \ref{thm:chain_coefficients}.

\begin{example}\label{ex:narayana}
The following computation using Theorem \ref{thm:h*_ladder} yields the same result as in Example \ref{example: using section 3 formula}:
    \begin{align*}
        h^*(\epsilon R; z) &= N(3,1)z^0 + N(3,2)z^1 + N(3,3)z^2\\
        &= 1 + 3z + z^2.
    \end{align*}
\end{example}

Knowing the structure of the $h^*$-polynomial of $\mathcal{O}(P(\bw))$, allows us to obtain the form of the Ehrhart polynomial of $\mathcal{O}(P(\bw))$, as the following shows for the ladder.

\begin{corollary} \label{cor: ehrhart ladder}
The Ehrhart polynomial of $\mathcal{O}(P(\bw))$ when $\bw$ is a ladder, i.e., $\bw= \epsilon LLLL\cdots$ or $\epsilon RRRR \cdots$ of length $n$, is of the form
        \[
             L(\bw; t) = \frac{1}{(n + 2)!(n + 3)!}(t + 1)(t + 2)^2(t+3)^2\cdots (t + n +2 )^2(t + n +3).
        \]
    \label{formula: ehrhart of ladder}
\end{corollary}
\begin{remark}
Note that Corollary \ref{cor: ehrhart ladder} can be proved without relying on the formula for $h^*(\bw; z)$.
A sketch of this alternate proof is as follows. 
Consider \[\mathcal{B} (r, s,t):= \{((i,j,k)\;|\; 1\leq i \leq r, 1\leq j \leq s, 1\leq k \leq t)\},\] then the number of plane partitions in $\mathcal{B}(r,s,t)$ is given in \cite{MR0141605} by \[N_1(r,s,t) = \prod\limits_{i = 1}^r \prod\limits_{j = 1}^s \frac{i + j + t - 1}{i + j -1 }.\]
One can then biject the integer points in $t\calO(P(\bw))$ with plane partitions in $\mathcal{B}(2, n + 2, t)$, in other words, $L(\bw; t) = N_1(2, n +2, t) $.
Then applying MacMahon's formula we obtain $L(\bw; t)$.
\end{remark}

\subsection{\texorpdfstring{$h^*$}{}-polynomial recurrence}
We prove a recurrence formula for the $h^*$-polynomial in this section.

We proceed by first establishing some extra definitions and notation. 
Take $\bw$ to be a generalized snake word of length $n$.  
\begin{itemize}
    \item Let $\bw[i:j]$ denote the subword of $\bw$ from the $i^{\text{th}}$ index to the $j^{\text{th}}$. Note that $\bw[i:j]$ is not always a generalized snake word, as the first letter is not $\epsilon$ unless $i=1$.

    \item We say that $\bw[i:j]$ for $i \leq j$ is a \defterm{maximal sub-ladder}, if and only if $w_{i-1} \neq w_i$ and $w_j \neq w_{j+1}$, and $\bw[i:j]$ consists of the same letter.

    \item Denote the subword $\bw[0:n-1]$ by $\bw'$.
    \item Denote the subword $\bw[0:i]$ by $\bw^{*}$ if $\bw[i+1:n]$ is a maximal sub-ladder.

\end{itemize}

\begin{example}
Consider the generalized snake word $\bw = \epsilon RRLLLRRR$. 
Then, \[\bw' = \epsilon RRLLLRR \, \text{ and }\,  \bw^* = \epsilon RRLLL.\]
\end{example}

\begin{theorem} \label{thm: recurrence for h*}
Consider a generalized snake word  $\bw$ and take $\bq$ to be $\epsilon$ concatenated by the right-most maximal sub-ladder of $\bw$. 
Then the following recurrence relation holds.
    \[
        h^{*}(\bw;z) = h^{*}((\bw^{*})';z) \cdot h^{*}(\bq;z) + h^{*}(\bw^{*};z) \cdot h^{*}(\bq';z) - (z+1) \cdot h^{*}((\bw^{*})';z) \cdot h^{*}(\bq';z)
    \]
Take the $h^{*}$-polynomials of ladders as base cases and take $h^{*}(\epsilon; z) := z+1$.
\end{theorem}

\begin{proof}
Let $\bw = w_0w_1 \dots w_n$ with $w_0 = \epsilon$ and take $k$ to be the maximal index such that $w_k \neq w_n$.
Denote the set of maximal chains of $J(P(\bw))$ that contain $i$ red L turns by $\mathcal{M}_i(\bw)$ (recall that these chains correspond to linear extensions of $P(\bw)$ with $i$ ascents).

Denote the set of maximal chains of $\mathcal{M}_i(\bw)$ that include elements of $J(P(\bw))$ that are labeled with $\langle a \rangle$ or $\langle a, b \rangle$, by $\mathcal{M}_i^{\langle a\rangle}(\bw)$ and $\mathcal{M}_i^{\langle a, b\rangle}(\bw)$, respectively. 
Denote the set of maximal chains of $\mathcal{M}_i(\bw)$ that includes elements of $J(P(\bw))$ that are less than or equal to (greater than equal to) the element labeled with $\langle a \rangle$ as $\mathcal{M}_i^{\langle a\rangle\downarrow}(\bw)$ ($\mathcal{M}_i^{\langle a\rangle\uparrow}(\bw)$).
Use an analogous notation for the case where we take the element labeled with $\langle a, b \rangle$ instead of $\langle a \rangle$.
    Then, for $k \geq 1$, we have the following by inclusion-exclusion,
    \[
        \mathcal{M}_i(\bw) = \mathcal{M}_i^{\langle 2k+1, 2k+2 \rangle}(\bw) \cup \mathcal{M}_i^{\langle 2k, 2k-1 \rangle}(\bw) - \left(\mathcal{M}_i^{\langle 2k+1, 2k+2 \rangle}(\bw) \cap \mathcal{M}_i^{\langle 2k, 2k-1 \rangle}(\bw)\right).
    \]
    Therefore, 
    \begin{align*}
        |\mathcal{M}_i(\bw)| &=  \sum_{j=0}^i |\mathcal{M}_j^{\langle 2k+1, 2k+2 \rangle\downarrow}(\bw)||\mathcal{M}_{i-j}^{\langle 2k+1, 2k+2 \rangle\uparrow}(\bw)| + \sum_{j=0}^i |\mathcal{M}_j^{\langle 2k, 2k-1 \rangle\downarrow}(\bw)||\mathcal{M}_{i-j}^{\langle 2k, 2k-1 \rangle\uparrow}(\bw)|\\
        &- \sum_{j=0}^i |\mathcal{M}_j^{\langle 2k+1, 2k+2 \rangle\downarrow}(\bw)||\mathcal{M}_{i-j}^{\langle 2k, 2k-1 \rangle\uparrow}(\bw)| - \sum_{j=0}^{i-1} |\mathcal{M}_j^{\langle 2k+1, 2k+2 \rangle\downarrow}(\bw)||\mathcal{M}_{i-j-1}^{\langle 2k, 2k-1 \rangle\uparrow}(\bw)|\\
        &= \sum_{j=0}^i |\mathcal{M}_j(\bw[k:n])||\mathcal{M}_{i-j}(\bw[0:k-1])| + \sum_{j=0}^i |\mathcal{M}_j(\bw[k+1:n])||\mathcal{M}_{i-j}(\bw[0:k])|\\
        &- \sum_{j=0}^i |\mathcal{M}_j(\bw[k+1:n])||\mathcal{M}_{i-j}(\bw[0:k])| - \sum_{j=0}^{i-1} |\mathcal{M}_j(\bw[k+1:n])||\mathcal{M}_{i-j-1}(\bw[0:k])|,
    \end{align*}
where the first equality follows from the construction of red L turns.
Combining Lemma \ref{lemma: Eulerian polynomial} and Lemma \ref{lemma: red L turn}, we have that $[z^i]h^*(\bw) = |\mathcal{D}_i(\bw)|  =|\mathcal{M}_i(\bw)|$. Thus, we deduce the following:
    \begin{align*}
        h^{*}(\bw;z) = h^{*}(\bw[0:k-1];z) \cdot h^{*}(\bw[k:n];z) &+ h^{*}(\bw[0:k];z) \cdot h^{*}(\bw[k+1:n];z)\\
        &- (z+1) \cdot h^{*}(\bw[0:k-1];z) \cdot h^{*}(\bw[k+1:n];z).
    \end{align*}
Since $k$ is the maximal index such that $w_k \neq w_n$, we further obtain
\[
h^{*}(\bw;z) = h^{*}((\bw^{*})';z) \cdot h^{*}(\bq;z) + h^{*}(\bw^{*};z) \cdot h^{*}(\bq';z) - (z+1) \cdot h^{*}((\bw^{*})';z) \cdot h^{*}(\bq';z)
\]
\end{proof}

The following example illustrates the use of the recurrence formula to compute the $h^*$-polynomial for any generalized snake word $\bw$.

\begin{example}
Consider $\bw = \epsilon RRL$ and note that $\bw^* = \epsilon RR, \bq = \epsilon L$.
Then
    \[
        h^*(\bw, z) = h^*(\epsilon R, z) h^*(\epsilon L, z) + h^*(\epsilon RR, z) h^*(\epsilon, z) - (z + 1) h^*(\epsilon R, z) h^*(\epsilon, z).
    \]
Hence, the computation is reduced to the $h^*$-polynomial of the ladder, which we already know from Example \ref{example: using section 3 formula}/Example \ref{ex:narayana}:
\[h^*(\epsilon R, z) = h^*(\epsilon L, z) = z^2 + 3z + 1, \; h^*(\epsilon RR, z) = z^3 + 6 z^2 + 6 z + 1.\]
Therefore, the $h^*$-recurrence formula gives that:
\begin{align*}
    h^*(\bw, z)
    &= h^*(\epsilon R, z) h^*(\epsilon L, z) + h^*(\epsilon RR, z) h^*(\epsilon, z) - (z + 1) h^*(\epsilon R, z) h^*(\epsilon, z) \\
    &= (z^2 + 3z + 1)^2 + (z + 1) \cdot (z^3 + 6 z^2 + 6 z + 1) - (z + 1)^2 \cdot (z^2 + 3z + 1)\\
    &= z^4 + 8 z^3 + 15 z^2 + 8 z + 1,
\end{align*}
which aligns with our computation using Sage \cite{Sage} and is not accounted for by either Theorem \ref{thm: h*_snake} or Theorem \ref{thm:h*_ladder}.
\end{example}
\begin{remark}
By substituting $z=1$ to Theorem \ref{thm: recurrence for h*} and using Theorem \ref{thm:h*_ladder}, we recover the recurrence relation of volumes in Theorem \ref{thm:O(P(w))Volume} \cite[Theorem 3.6]{vonBell+}.
\end{remark}

\subsection{Swap operation and monotonicity result}
Consider the following $h^*$-polynomials:
\[
h^*(\epsilon RRR, z) = z^4 + 10 z^3 + 20 z^2 + 10 z + 1,\]
\[
h^*(\epsilon RRL, z) = z^4 + 8 z^3 + 15 z^2 + 8 z + 1, \text{ and }
\]
\[ 
h^*(\epsilon RLR, z) = z^4 + 7 z^3 + 13 z^2 + 7 z + 1.
\]
Observe the coefficient-wise inequality:
\[
[z^k] h^*(\epsilon RLR, z) \leq [z^k] h^*(\epsilon RRL, z) \leq [z^k] h^*(\epsilon RRR, z).
\] 
Our goal in this section is to generalize this inequality using the \textit{swap operator}.
We prove a coefficient-wise monotonicity result in Theorem \ref{thm: coef mono} using the swap operator. 
Our result is strengthening of Theorem \ref{thm:minmaxvolumes} \cite[Theorem 3.10]{vonBell+}. 
In fact, we can retrieve Theorem \ref{thm:minmaxvolumes} by summing the coefficient inequality from $k = 0$ to $\deg(h^*)$.

\begin{definition}
Define the \defterm{swap operation} as
    \[
        \varphi_i: \left\{\text{generalized snake words of length } n\right\} \longrightarrow \left\{\text{generalized snake words of length } n\right\}, 
    \]
    where $\varphi_i(\bw)$ is the generalized snake word obtained from $\bw$ by swapping all letters with indices greater than or equal to $i$ from $R$ to $L$ and $L$ to $R$. 
    The word $\bw$ is indexed from left to right in increasing order with the empty word $\epsilon$ having index $1$.
\end{definition}

\begin{example}
Applying the swap operation to $\bw = \epsilon RRLLLRRRRLL$ with $i=7$, we obtain that $\varphi_7(\bw) = \epsilon RRLLLLLLLRR$.

\end{example}
\begin{theorem} \label{thm: coef mono}
For any generalized snake word $\bw$ of length $n>0$ and for all $k \geq 0$, the following inequality holds:
\[[z^k]h^*(\varphi_i(\bw); z) \leq [z^k]h^*(\bw; z),\]
whenever $w_{i-1} = w_i$ or $i = 1$. 
Equality occurs only when $i = 1$.
\end{theorem}

\begin{figure}[!ht]
    \centering
    \scalebox{1}{\input{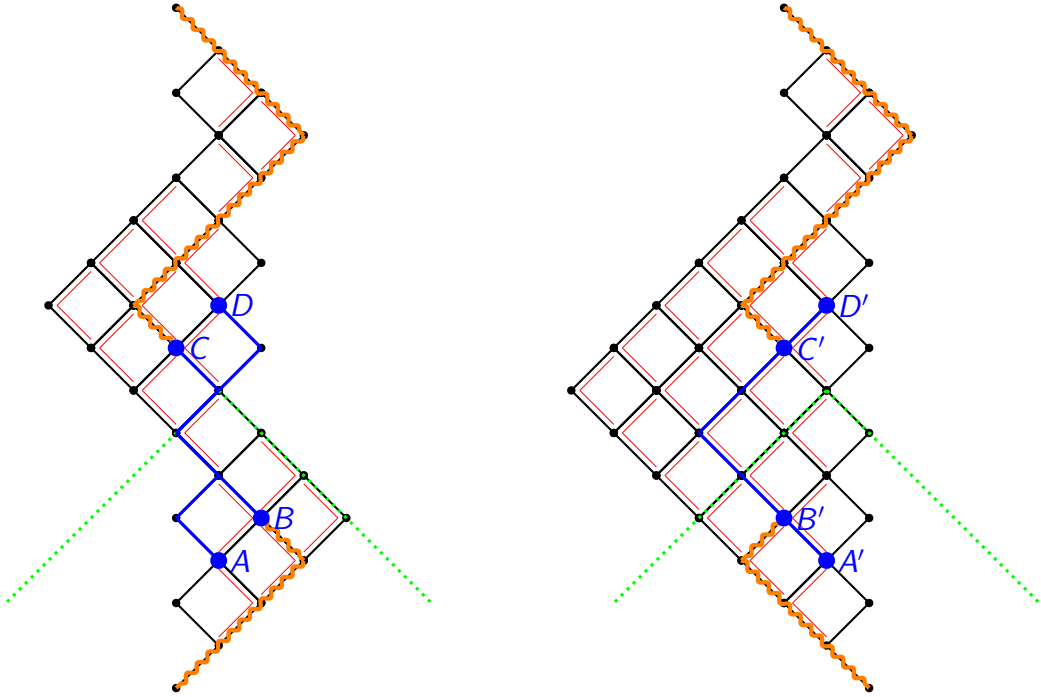}
}
    \caption{Injective map $\phi_5$ mapping $\epsilon RLLLRR$ to $\epsilon RLLLLL$.}
    \label{fig:injective_map}
\end{figure}

\begin{proof}
Consider a generalized snake word $\bw$ of length $n>0$ and let $i \neq 1$ and $w_{i-1} = w_i$.
Similar as before, denote the set of maximal chains of $J(P(\bw))$ that corresponds to the linear extensions of $P(\bw)$ with $k$ ascents by $\mathcal{M}_k(\bw)$.
The maximal chains of $J(P(\bw))$ can be partitioned as 
\[ \mathcal{M}(\bw) = \mathcal{M}_0(\bw) \cup \mathcal{M}_1(\bw) \cup \dots \cup \mathcal{M}_{n+1}(\bw).\]
In what follows, we construct an injective map \[\phi_i:\mathcal{M}(\varphi_i(\bw)) \longrightarrow \mathcal{M}(\bw),\] such that $\mathcal{M}_k(\varphi_i(\bw)) \subseteq \phi_i^{-1}(\mathcal{M}_k(\bw))$ for all $k \geq 0$.
    
Without loss of generality, assume that $w_i = w_{i-1} = L$.
Let \[A = \langle2i+1, 2i+2 \rangle, \;B = \langle2i-1, 2i+2\rangle,\; C = \langle2i-5, 2i-2\rangle,\; D = \langle2i-5, 2i-4 \rangle\] be elements of $J(P(\varphi_i(\bw))$ and \[A' = \langle 2i+1, 2i+2 \rangle, \; B' = \langle 2i-1, 2i+2 \rangle, \; C' = \langle 2i-5, 2i-2 \rangle, \; D' = \langle 2i-5, 2i-4 \rangle\] be elements of $J(P(\bw))$.

Take $\calC_1 = \{ \langle2i-1, 2i\rangle \langle2i-3, 2i\rangle \langle2i-3, 2i-2\rangle \}$.
Then consider the following chains in $J(P(\varphi_i(\bw))$:
    \begin{align*}
        \calC_{AC} &= \{A, \langle2i+1\rangle \} \cup \calC_1 \cup \{C\},\\
        \calC_{AD} &= \{A, \langle2i+1\rangle \} \cup \calC_1 \cup \{\ \langle2i-3\rangle, D\},\\
        \calC_{BC} &= \{B\} \cup \calC_1 \cup \{C\}, \text{ and }\\
        \calC_{BD} &= \{B\} \cup \calC_1 \cup \{\langle2i-3\rangle, D\}.
    \end{align*}

Next, take $\calC_2 = \{ \langle2i-5, 2i\rangle \langle2i-5, 2i+2\rangle \langle2i-3, 2i+2\rangle \}$ and consider the following chains in $J(P(\bw))$:
    \begin{align*}
        \calC_{A'C'} &= \{A', B'\} \cup \calC_2 \cup \{C'\},\\
        \calC_{A'D'} &= \{A', B'\} \cup \calC_2 \cup \{C', D'\},\\
        \calC_{BC} &= \{B'\} \cup \calC_2 \cup \{C'\}, \text{ and }\\
        \calC_{BD} &= \{B'\} \cup \calC_2 \cup \{C', D'\}.
    \end{align*}

Now, consider the maximal chains of $J(P(\varphi_i(\bw))$ that do not include $\calC_{XY}$ for $X \in \{A, B\}$ and $Y \in \{C, D\}$.
Then, these chains in $J(P(\varphi_i(\bw))$ get mapped to chains in $J(P(\bw))$ by reflecting the portion of the chain in $J(P(\varphi_i(\bw))$ consisting of element less than or equal to $\langle2i-3, 2i-2\rangle$ and preserving the remaining portion. 

Next, consider the paths that include $\calC_{XY}$ for $(X, Y) \in \{(A, C), (A, D), (B, C), (B, D)\}$.
Then, these chains in $J(P(\varphi_i(\bw))$ get mapped to chains in $J(P(\bw))$ by reflecting the portion of the chain in $J(P(\varphi_i(\bw))$ consisting of element less than or equal to $A, A, B, B$, respectively, mapping the path $\calC_{XY}$ in $J(P(\varphi_i(\bw))$ to the path $\calC_{X'Y'}$ in $J(P(\bw))$, respectively, and preserving the remaining portion.
[\emph{For example, in Figure \ref{fig:injective_map}, we can check that the path in the left Hasse diagram consisting of two paths colored in orange and a blue connecting the two orange paths from point $B$ to $C$, is mapped to the path in the right Hasse diagram consisting of two paths colored in orange and a blue path connecting the orange paths from point $B'$ to $C'$. 
Also, notice that the paths in the left and right diagram pass the same number of red $L$-turns.}]

Note that the map $\phi_i$ is an injective map that preserves the number of red $L$-turns that each maximal chain contains, i.e., \[\mathcal{M}_k(\varphi_i(\bw)) \subseteq \phi_i^{-1}(\mathcal{M}_k(\bw)) \text{ for all } k \geq 0.\]

As $|\mathcal{M}_k(\bw)| = [z^k]h^*(\bw;z)$ for all $k \geq 0$, we have that $[z^k]h^*(\varphi_i(\bw); z) \leq [z^k]h^*(\bw; z)$ for all $k \geq 0$.
\end{proof}

Specifically, we can bound the coefficients of the $h^*$-polynomial of any generalized snake word by that of the $h^*$-polynomial of the ladder and that of the $h^*$-polynomial of the regular snake.
\begin{corollary}\label{cor 4.20}
    For any generalized snake word $\bw$ of length $n>0$. Let $\mathbf{s}$, $\mathbf{l}$ be the ladder and the regular snake word of the same length, respectively. Then for all $k \geq 0$, the following inequalities hold:
$$
[t^k]h^*(\mathbf{s}; t) \leq [t^k]h^*(\mathbf{w}; t) \leq [t^k]h^*(\mathbf{l}; t)
$$
\end{corollary}

\begin{proof}
    For any $\bw$, We can obtain $\bw$ by applying a series of swap operations to $\bl$. Similarly we can obtain $\bs$ by applying a series of swap operations to $\bw$. Therefore, by applying Theorem \ref{thm: coef mono}, we can derive the corollary above.
\end{proof}

\begin{remark}
We note that Theorem \ref{thm: coef mono} and Corollary \ref{cor 4.20} resemble Stanley's monotonicity theorem \cite{Stanley_monotonicity}, which states that if a polytope $Q$ contains polytope $P$, then the $h^*$-vector of $Q$ is component-wise greater than or equal to the $h^*$-vector of $P$.
While the results presented above bear resemblance to one another, they do not follow from Stanley's result. 
In particular, the proof of Theorem \ref{thm: coef mono} builds an injective map that is defined using the maximal chains of $J(P)$, where a snake word of length $n$ gets mapped to another word of the same length, which would not fall under the purview of Stanley's monotonicity theorem since the order polytopes arising from each of the generalized snake words would be of the same dimension and would not be completely contained in one another.
\end{remark}

\section{Future directions}\label{sec:conclusion}
To conclude, we present some conjectures for further investigation based on our computations. 

\begin{conjecture} \label{conj 1}  Let $\bw$ be a generalized snake word of length $n + 1$,
\begin{enumerate}
    \item All roots of $h^*(\bw; z)$ are real and negative. 
    \item All roots of $L(\bw; t)$ are contained in the disk $|z - \frac{n + 4}{2}| \leq \frac{n + 2}{2}, z \in \mathbb{C}$ with axis of symmetry $x = \frac{- n -4}{2}$.
\end{enumerate}    
\end{conjecture}

\noindent We have verified \ref{conj 1} for snake words of length up to $9$. 
\bigskip

We also observed \textit{Ehrhart positivity} empirically.
\begin{conjecture}\label{conj 2}
The order polytope  $\calO(P(\bw))$ is Ehrhart positive, i.e., $L(\bw; t)$ has nonnegative coefficients.
\end{conjecture}

We have also observed a similar coefficient-wise inequality for the Ehrhart polynomial under the swap operator, parallel to Theorem \ref{thm: coef mono}.
\begin{conjecture}\label{conj 3}
    For  $k > 0$, the following coefficient-wise inequality holds:
\[
    [t^k] L(\varphi_i(\bw); t) \leq [t^k] L(\bw, t).
\]
\end{conjecture}

 One plausible way of proving Conjecture \ref{conj 2}, is to prove the Ehrhart positivity of the regular snake case using Theorem \ref{thm: h*_snake}, then prove Conjecture \ref{conj 3}.

We also propose to generalize the operator $\varphi_i$ to other graded posets, and to analyze the behavior of the $h^*$-vectors or the Ehrhart polynomials of other order polytopes under the generalized $\varphi_i$; the aim is to ascertain whether analogous monotonicity properties persist in those cases.


\section*{Acknowledgments}
The authors thank Per Alexandersson, Matias von Bell, Benjamin Braun, and Tom Roby for fruitful conversations, comments, and/or for pointing us to relevant literature.
EL was supported by Basic Science Research Program through the National Research Foundation of Korea (NRF) funded by the Ministry of Education (NRF-2022R1F1A1063424). 
ARVM was partially supported by the National Science Foundation under Award DMS-2532321.


\bibliographystyle{amsplain}
\bibliography{references}


\end{document}